\documentclass[10pt, a4paper]{amsart}
 \usepackage{latexsym, amsmath, amssymb, longtable, booktabs,amscd,microtype,booktabs,cases}
\usepackage[square]{natbib}
\usepackage{graphicx}
\usepackage[dvipsnames]{xcolor}
\usepackage{caption}
\usepackage{tcolorbox}
\usepackage{dsfont}
\usepackage[all]{xy}
\usepackage{float}
\usepackage{enumerate}
\usepackage{multirow}
\usepackage{tikz}
\usetikzlibrary{matrix,arrows,decorations.pathmorphing}
\usepackage{collectbox}
\usepackage{array}
\usepackage{amsxtra}
\usepackage{mathrsfs}
\usepackage[colorinlistoftodos,prependcaption, textsize=tiny]{todonotes}
\reversemarginpar

\usepackage{enumitem}

\usepackage{hyperref}
\hypersetup{
	colorlinks,
	citecolor=purple,
	filecolor=black,
	linkcolor=blue,
	urlcolor=black
}
\bibpunct{[}{]}{,}{n}{}{;}

\numberwithin{equation}{section}
\usepackage{mathtools}

\DeclarePairedDelimiterX{\infdivx}[2]{}{}{%
	#1\;\delimsize\|\;#2%
}

\newcommand\Frob{{\mathrm{Frob}}}

\newcommand{\Z}{\mathbb{Z}}
\newcommand{\F}{\mathbb{F}}

\newcommand{\tr}{\operatorname{tr}}

\newcommand{\R}{\mathbb{R}}

\newcommand\A{\mathbb{A}}
\newcommand\C{\mathbb{C}}

\newcommand\GL{{\mathrm{GL}}}

\newcommand{\Q}{\mathbb{Q}}

\newcommand\Gal{{\mathrm{Gal}}}

\newcommand\Ind{{\mathrm{Ind}}}

\newcommand{\val}{\mathrm{val}}

\def\id#1{{\mathfrak{#1}}}

\newcommand{\mco}{\mathcal O}

\makeatother

\newtheorem{thm}{Theorem}[section]

\newtheorem{prop}[thm]{Proposition}
\newtheorem{lemma}[thm]{Lemma}
\theoremstyle{definition}
\newtheorem{definition}[thm]{Definition}

\newtheorem{remark}[thm]{Remark}

\theoremstyle{definition}

\theoremstyle{remark}

\theoremstyle{remark}

\theoremstyle{plain} 
\newcommand{\thistheoremname}{}
\newtheorem{genericthm}[thm]{\thistheoremname}

\newtheorem*{genericthm*}{\thistheoremname}
\newenvironment{namedthm*}[1]
{\renewcommand{\thistheoremname}{#1}%
	\begin{genericthm*}}
	{\end{genericthm*}}

\makeatletter
\def\imod#1{\allowbreak\mkern10mu({\operator@font mod}\,\,#1)}
\makeatother

\title{Towards a Generalized Maeda Conjecture for Modular Forms with Quadratic Nebentypus}

\author{Debargha Banerjee}
\address{Department of Mathematics, Indian Institute of Science Education and Research, Pune, India}
\email{debargha@iiserpune.ac.in}

\author{Dhrubajyoti Das}
\address{Department of Mathematics, Indian Institute of Science Education and Research, Pune, India}
\email{dhrubajyoti.das@students.iiserpune.ac.in}

\author{Srijan Das}
\address{Department of Mathematics, Indian Institute of Science Education and Research, Pune, India}
\email{das.srijan@students.iiserpune.ac.in}

\author{Tathagata Mandal}
\address{Department of Mathematical Sciences, Adamas University, Kolkata, India}
\email{math.tathagata@gmail.com}

\author{Sudipa Mondal}
\address{Department of Mathematics, Indian Institute of Technology Madras, Chennai, India}
\email{sudipa.mondal123@gmail.com}

\begin{document}
\date{\today}

\begin{abstract}
We establish a lower bound for the number of non-CM Galois orbits of newforms in $S_k(N,\Psi)$ with non-trivial quadratic nebentypus $\Psi$ for sufficiently large weights. Extending the work of Dieulefait, Pacetti, and Tsaknias in the trivial nebentypus setting, we analyze the restrictions imposed by the quadratic character on local inertial types and determine the number of admissible Galois orbits of such types. We further prove that Atkin-Li pseudo-eigenvalues are Galois equivariant and hence, up to a natural equivalence relation, define a global Galois invariant. Together with existence results for newforms having prescribed local behavior, these invariants yield a lower bound for the number of non-CM Galois orbits by counting compatible pairs of local-global invariants. Finally, computations in small weights show that this lower bound is not always attained, indicating that certain local equivalences are not realized globally by Galois conjugation over the coefficient field of the modular form.
\end{abstract}

\maketitle

\setcounter{tocdepth}{1}
\tableofcontents

\section{Introduction}

Let $S_k(N, \Psi)$ denote the space of cusp forms of weight $k$, level $N$, and nebentypus $\Psi$. Let $\mathbb{Q}(\Psi)$ denote the cyclotomic field generated by the values of $\Psi$. The group $\mathrm{Gal}(\overline{\mathbb{Q}}/\mathbb{Q}(\Psi))$ acts naturally on these forms via their Fourier coefficients, partitioning the set of normalized Hecke eigenforms (newforms) into disjoint Galois orbits. A central problem is to understand the asymptotic behavior of the number of such orbits as the weight $k$ increases. For a given newform $f$, the size of its Galois orbit is equal to the degree of the coefficient field $[E_f:\Q]$, which is controlled by the factorization of the characteristic polynomial of a Hecke operator. For level $1$, Maeda's conjecture \cite{HidaMaeda} predicts that this Hecke polynomial is irreducible over $\mathbb{Q}$ for every weight $k \geq 16$. Consequently, the coefficient field has the maximal degree $\dim S_k(\mathrm{SL}_2(\mathbb{Z}))$, and all newforms in the space belong to a single Galois orbit. Substantial computational evidence supports Maeda's conjecture; the current verification extends to all weights $k \leq 14000$ by work of Ghitza and McAndrew \cite{Ghitza}.

For higher levels, the Hecke algebra splits into a direct sum of multiple fields, giving rise to several distinct Galois orbits. In this setting, Tsaknias \cite{tsaknias} proposed a generalization of Maeda's conjecture for arbitrary fixed level $N$ and trivial nebentypus. It predicts that the number of non-CM Galois orbits, denoted by $\mathbf{NCM}(N,k)$, stabilizes for sufficiently large weights $k$. He also provided computational evidence supporting this phenomenon.

In a recent foundational work, Dieulefait, Pacetti, and Tsaknias \cite{dpt} established a lower bound $\prod_{p \mid N}\mathbf{LO}(p^{\mathrm{val}_p(N)})$ for $\mathbf{NCM}(N,k)$ for all sufficiently large weights $k$ when $N$ is a prime power or square-free. This bound arises from two invariants imposed by global Galois conjugacy at primes dividing $N$: a local invariant given by the Galois orbit of local inertial types, and a local-global invariant given by the compatible minimal Atkin-Lehner sign. By counting admissible pairs of these invariants, they explicitly computed the value of $\mathbf{LO}(p^{\mathrm{val}_p(N)})$ for all primes $p$ dividing $N$. Based on computational data, they further conjectured \cite[Question 4.5]{dpt} that this lower bound gives the exact asymptotic count for prime-power or squarefree case, that is in these cases $\prod_{p \mid N}\mathbf{LO}(p^{\mathrm{val}_p(N)}) = \mathbf{NCM}(N,k)$ for all sufficiently large $k$.

In this paper, we extend this framework to spaces of newforms with non-trivial nebentypus $\Psi$. We develop the underlying local-global framework for broad classes of general nebentypus characters. To explicitly count the admissible local invariants, however, we restrict to quadratic nebentypus, which guarantees $\mathbb{Q}(\Psi) = \mathbb{Q}$ and allows us to study the absolute Galois orbits of these forms. Our main objective is to establish lower bounds for the number of such orbits in a setting where classical Atkin-Lehner theory is no longer available. More precisely, we show that the local inertial type and the minimal Atkin-Li pseudo-eigenvalue, considered up to a natural equivalence relation, give rise to global Galois invariants. We then determine the number of compatible pairs of these invariants and prove that every such pair is realized by a non-CM newform of sufficiently large weight.

To state this precisely, let $\mathbf{LO}(p^n, \Psi)$ denote the number of pairs consisting of a local-type Galois orbit of level $p^n$, quadratic nebentypus $\Psi$ and a compatible equivalence class of minimal Atkin-Li pseudo-eigenvalues (See Definition \ref{def:compatibleAL} and subsequent paragraph for the precise definition). Furthermore, let $\mathbf{NCM}(N,k,\Psi)$ denote the number of Galois orbits of non-CM newforms in $S_k(N, \Psi)$. Our main result is the following lower bound.

\begin{thm} \label{main_ineq}
Let $N$ be a positive integer such that either $N$ is a prime power, or $\val_p(N)$ is an odd integer $\geq 3$ for every prime $p \mid N$, and let $\Psi$ be a quadratic Dirichlet character modulo $N$. Then there exists a positive integer $k_0$ such that for any $k \geq k_0$,
\begin{equation} \label{main}
\prod_{p \mid N} \mathbf{LO}(p^{\val_p(N)}, \Psi) \leq \mathbf{NCM}(N,k,\Psi).
\end{equation}
\end{thm}

We restrict our counting to non-CM forms because the number of CM Galois orbits is periodic in the weight $k$, preventing the total orbit count from eventually becoming constant. Furthermore, for sufficiently large weights, every local type realized by a CM form is simultaneously realized by a non-CM form of the same level, meaning that the asymptotic behavior of the total number of Galois orbits is entirely determined by $\mathbf{NCM}(N, k, \Psi)$. 

For further background on Maeda's conjecture and its arithmetic consequences, we refer to \cite{Maeda}. Additional evidence for Maeda-type phenomena is provided by the analytic and average-case results of Murty-Srinivas \cite{MurtySrinivas} and Martin \cite{Martin}. Recently Lam \cite{joshlam} has established results towards a function-field analogue of the conjecture.

\subsection*{Organization of the paper}
We now outline the structure and main results of the paper. Section \ref{sec:prelim} establishes the necessary background and notation on the local Langlands correspondence and inertial types for $\GL_2$.

In Section \ref{localtypegaloisorbit}, we study Galois conjugacy classes of local inertial types. Since we consider types arising from modular forms with quadratic nebentypus, the central character condition imposes strong local constraints. Because a quadratic character has conductor at most $1$, the analysis separates according to whether the local nebentypus $\Psi_p$ is unramified or tamely ramified. This classification leads to explicit formulas for the number of admissible local type orbits $\mathbf{LT}(p^n, \Psi)$, given in Theorem \ref{thm:LTformula_nebentypus}.

Section \ref{sec: types-of-modular-forms} establishes the global Galois invariants required for our counting formulas. We first prove in Lemma \ref{prin} and Lemma \ref{sup} that if a newform realizes a given local type, then its Hecke field $E_f$ is sufficiently large to realize all elements in the Galois orbit of that type. This guarantees that organizing local types by Galois conjugacy is compatible with the study of global Galois orbits, which we formalize in Theorem \ref{orbit-preserves-type}.

We then introduce the Atkin-Li pseudo-eigenvalue as our second invariant. A key issue, highlighted by Equation \ref{eq:imp-of-min-twist}, is that under highly ramified twists, the pseudo-eigenvalue becomes entirely determined by the twisting character and the nebentypus, thereby losing any information about the original Galois orbit. To extract a meaningful invariant, we restrict our attention to the pseudo-eigenvalue of a minimal quadratic twist. We then determine how these pseudo-eigenvalues behave under Galois action by relating them to Langlands $\varepsilon$-factors in Theorem \ref{pseudo}. We show that they are not invariant under the Galois orbit of modular forms, but rather Galois equivariant up to an explicit nebentypus twist. This necessitates the introduction of an equivalence relation (defined in \ref{equiv-reln-on-al-ev}), under which the pseudo-eigenvalue becomes a well-defined global invariant. Although we focus on quadratic nebentypus, this formalism extends to general nebentypus, as explained in Remark \ref{al-general-nebentypus}.

A central component of our work is determining which Atkin-Li pseudo-eigenvalues can occur for a fixed local type, since enumerating these valid combinations is precisely what yields our global lower bound. This classification is established in Theorem \ref{CountofLocalType}. In many cases, the local type uniquely determines the pseudo-eigenvalue. In the remaining cases, there are two possibilities differing by a sign. In the trivial nebentypus case, the Atkin-Lehner sign is an invariant of the Galois orbit, so the counting problem is entirely local and leads to an explicit formula. In our setting, however, the pseudo-eigenvalue is defined only up to an equivalence relation arising from the Galois action. As a result, the two values $\lambda$ and $-\lambda$ may or may not define the same invariant. This creates an obstruction to writing a completely explicit formula. The local analysis alone does not determine whether the two values should be identified; this depends on the global Galois action.

To address this issue, we separate the relevant local types into two classes. In the first case, the two values become equivalent under Galois action and hence contribute a single invariant. In the second case, they remain distinct and contribute two invariants. We refer to these as symmetric and asymmetric orbits, respectively. The above distinction allows us to express the count $\mathbf{LO}(p^n,\Psi)$ in a clean form (cf. Theorem \ref{count-of-LO}), while isolating the contribution arising from asymmetric orbits. Although this symmetric-versus-asymmetric dichotomy can be determined explicitly for the Steinberg representation, the Gauss sum required for computing the local factor of ramified supercuspidal types often remains intractable. Since we cannot explicitly compute these pseudo-eigenvalues to determine whether they are conjugate under the global Galois action, we retain the asymmetric count $|S_{\mathrm{asym}}|$ as a parameter in our final enumerative formulas.

In Section \ref{sec:existence-thm}, we prove that these local configurations are globally realizable. Using limit multiplicity results of Weinstein (for principal series) and Kim, Shin, and Templier (for discrete series), we show in Theorem \ref{existence-thm} that every admissible pair arises from a non-CM newform for all sufficiently large weights, thereby establishing the lower bound for $\mathbf{NCM}(N,k,\Psi)$ stated above. To apply these limit multiplicity formulas, we require the level $N$ either to be a prime power or to satisfy the condition that $\mathrm{val}_p(N)$ is an odd integer at least $3$ for every prime $p \mid N$. If this existence result is proved for arbitrary $N$, then our lower bound theorem will hold in full generality.

Finally, Section \ref{sec:computational_evidence} presents computational evidence drawn from the LMFDB illustrating the counting formulas and the role of symmetric and asymmetric orbits. For several prime power levels where the local types are ramified supercuspidal, we compare the count $\mathbf{LO}(p^n,\Psi)$ with the observed number of non-CM Galois orbits $\mathbf{NCM}(p^n,k,\Psi)$. The data show that strict inequality $\mathbf{LO} < \mathbf{NCM}$ can occur at small weights. This reflects the fact that sometimes local invariants are identified under action of $\Gal(\overline{\Q}/\Q)$, which the corresponding global Hecke fields may not be large enough to realize. As a result, locally equivalent pseudo-eigenvalues can remain distinct at the global level.
\medskip

\textbf{Acknowledgements.}  D. B. is supported by ANRF grant ANRF/ARGM/2025/000421/MTR.
S.D. is supported by Prime Minister's Research Fellowship, Government of India. D. D. is supported by DST-Inspire fellowship, Government of India. S.M. acknowledges the support provided by the Institute Postdoctoral fellowship at IIT Madras.
\section{Preliminaries}
\label{sec:prelim}
Let $f =\sum a_n q^n \in S_k(N,\Psi)$ be a newform of weight $k \geq 2$, level $N$, and quadratic character $\Psi$. For a prime $p \mid N$, let $\mathbf{NCM}(N,k,\Psi)$ denote the number of Galois orbit of non-CM newforms in $S_k(N,\Psi)$. The primary aim of this project is to establish a lower bound on $\mathbf{NCM}(N,k,\Psi)$ for certain values of $N$.

Let $\pi_f$ denote the automorphic representation of the ad\`ele group $\mathrm{GL}_2(\A_\Q)$ attached to $f$. It is well known that $\pi_f$ decomposes as a restricted tensor product $\pi_f \cong \bigotimes_v' \pi_{f,v},$ where $v$ runs over all places of $\Q$. Here $\pi_{f,v}$ is an irreducible admissible representation of $\mathrm{GL}_2(\Q_v)$ and $\Q_v$ denotes the completion of $\Q$ at $v$.

Let $K$ be a non-archimedean local field of characteristic zero. Denote by $\mathcal{O}_K$ its ring of integers, by $\mathfrak{p}$ the maximal ideal, and by $k$ the residue field, of cardinality $q$. For $n \ge 1$, we write $U_K^n = 1 + \mathfrak{p}^n$ for the group of principal units of level $n$. For $n=0$, we have $U_K^0=\mco_K^\times$.

\begin{definition}
	The conductor $a(\chi)$ of a multiplicative character $\chi$ of $K^\times$ is the smallest non-negative integer $n$ such that $\chi|_{{U_K^n}} = 1$.	The conductor of a non-trivial additive character $\phi$ of $K$ is an integer $n(\phi)$ such that $\phi$ is trivial on $\mathfrak{p}^{-n(\phi)}$, but non-trivial on $\mathfrak{p}^{-n(\phi)-1}$.\end{definition}
A character $\chi$ is called \emph{unramified} if 
 	the conductor is zero,
 	\emph{tamely ramified} if it has conductor $1$ and
 	\emph{wildly ramified} if its conductor is greater or equal to $2$.
 	For any two characters $\chi_1$ and $\chi_2$, we have
 	$a(\chi_1\chi_2) \leq \text{max}(a(\chi_1),a(\chi_2))$.
 	The equality holds if $a(\chi_1) \neq a(\chi_2)$.

The following two lemmas are used repeatedly in the computations of Section~\ref{localtypegaloisorbit}. For convenience, we state and prove them separately here.


\begin{lemma} \label{Hensel}
    Let $p$ be a prime and $d$ is a positive integer such that $p \nmid d$. Every element of the principal unit group $U_{\Q_p}^{(1)} = 1 + p\Z_p$ is a perfect $d$-th power in $1 + p\Z_p$.
\end{lemma}
\begin{proof}
    Let $x \in 1+p\Z_p$ and consider $f(Y)=Y^d-x \in \Z_p[Y]$. Since $x \equiv 1 \pmod p$, modulo $p$ the polynomial reduces to $\overline f(Y)=Y^d-1$, which admits the root $Y \equiv 1 \pmod p$. Moreover, $f'(Y)=dY^{d-1}$ gives $f'(1)=d \not\equiv 0 \pmod p$. By Hensel’s lemma, this root lifts uniquely to $y \in \Z_p$ with $y^d=x$ and $y \equiv 1 \pmod p$, hence $y \in 1+p\Z_p$.
\end{proof}

\begin{lemma} \label{L:quad-cond}
    Let $p$ be a prime and let $\chi: \Q_p^\times \to \C^\times$ be a character of order $d$ such that $p \nmid d$. Then the conductor of $\chi$ satisfies $a(\chi) \le 1$.
\end{lemma}

\begin{proof}
    To show $a(\chi)\le 1$, it suffices to prove that $\chi$ is trivial on $U_{\Q_p}^{(1)}=1+p\Z_p$. By Lemma \ref{Hensel} every $x\in 1+p\Z_p$ is a perfect $d$-th power in $1+p\Z_p$. As $\chi$ is of order $d$, the result follows.
\end{proof}

Let $K=\Q_p(\sqrt{d})$ be a quadratic extension of $\Q_p$ with its ramification index $e$. The following result \cite[Theorem 2.10]{dpt} determines the group structure of $(\mco_K/\id{p}^n)^\times$.

\begin{thm} \label{units-in-residue}
  Let $n$ be a positive integer and let $K/\Q_p$ be a quadratic extension for an odd prime $p$. The group structure of $(\mco_K/\id{p}^n)^\times$ is given in Table \ref{Table:units-in-residue} where ``--'' means no condition, and the pair $(a,b)$ satisfies the following two conditions (which determine them uniquely):
    \begin{itemize}
        \item  $a+b=n-1$,
        \item $a=b$ if $n$ is odd,
        \item $a=b+1$ if $n$ is even.
    \end{itemize}

\begin{table}[h]
\centering
\scalebox{0.95}{
\begin{tabular}{c c c c c c}
  \hline
$K$ & $e$ &$p$ &  $n$ & Structure & $\text{Generators}$\\
  \hline
--- & $1$ & $\neq 2$ & ---& $\F_{p^2}^\times \times \Z/p^{n-1} \times \Z/p^{n-1}$ & $\{\xi_{p^2-1},1+p,1+p\sqrt{d} \}$\\
  \hline
$\neq \Q_3(\sqrt{-3})$ & $2$ & $\neq 2$ & --- &$\F_p^\times \times \Z/p^a \times \Z/p^b$ & $\{\xi_{p-1},1+\sqrt{d},1+p\}$\\
  \hline
  $\Q_3(\sqrt{-3})$ & $2$ & $3$ & $\ge 2$ & $\F_3^\times \times \Z/3 \times \Z/3^{a-1} \times \Z/3^{b}$ & $\{-1,\xi_3,1+\sqrt{-3},4 \}$\\
  \hline
\end{tabular}
}

\caption{Group structure of $(\mco_K/\id{p}^n)^\times$ for $p \neq 2$.}
\label{Table:units-in-residue}
\end{table}
\end{thm}

\subsection{Inertial types for $\mathrm{GL}_2$}
\begin{definition} \label{local-global}
	\begin{enumerate}
		\item
		Let $\tau$ be a Weil–Deligne representation. By the local inertial type of $\tau$, denoted by $\widetilde{\tau}$, we mean the isomorphism class of its restriction to the inertia subgroup. A type is called unramified if $\widetilde{\tau}$ is the trivial representation.
		\item 
		A collection $(\widetilde{\tau}_p)_p$ of local inertial types, where $p$ runs over all primes,  is called a global inertial type if $\widetilde{\tau}_p$ is unramified for almost all primes.
	\end{enumerate}
\end{definition}
Let $\mathcal{A}_p$ denote the set of isomorphism classes of complex-valued irreducible admissible representations of $\GL_2(\Q_p)$.  Let $\mathcal{B}_p$ denote the set of isomorphism classes of  $2$-dimensional Frobenius-semisimple Weil–Deligne representations of $\Q_p$. By the local Langlands correspondence, we have a bijective correspondence between $\mathcal{A}_p$ and $\mathcal{B}_p$, 
\begin{equation} \label{bijective}  \pi \longleftrightarrow \tau(\pi).  \end{equation}
Using the above correspondence, one can study the local inertial type by the restriction $\pi|_{\GL_2(\Z_p)}$ (see \cite[Section 2.1]{weinstein} for more details).

\noindent \textbf{Elements of $\mathcal{A}_p$:} Set $G:=\mathrm{GL}_2(\Q_p)$.
For two quasi-characters $\mu_1,\mu_2$ of $\Q_p^\times$, let $V(\mu_1,\mu_2)$ denote the space of locally constant functions $\psi: G \to \C$ satisfying
\[
\psi \Big(
\left[ {\begin{array}{cc}
		a & * \\
		0 & d \\
\end{array} } \right]
g
\Big)
=
\mu_1(a) \mu_2(d) |a/d|^{1/2} \psi(g),
\]
for all $a,d \in \Q_p^\times$ and $g \in G$. The induced representation of $G$ by its action on $V(\mu_1,\mu_2)$ through right translation is denoted by $\pi(\mu_1,\mu_2)$. An element of $\mathcal{A}_p$ is one of the following:
\begin{enumerate} 
	\item \smallskip 
	If $\mu_1 \mu_2^{-1} \neq |\,\,|^{\pm 1}$, then $\pi(\mu_1,\mu_2)$ is irreducible and the representation $\pi(\mu_1,\mu_2)$ is called a \emph{principal series representation}. The central character of $\pi(\mu_1,\mu_2)$ equals $\mu_1 \mu_2$ and its conductor equals $a(\mu_1)+a(\mu_2)$.
	\item \smallskip 
	The unique irreducible sub-representation of $\pi(|\,\,|^{1/2}, |\,\,|^{-1/2})$ is the \emph{Steinberg representation}, denoted by St. More generally, we consider $\mu_1=\mu |\,\,|^{1/2}$, $\mu_2=\mu |\,\,|^{-1/2}$ for some character $\mu$ of $\Q_p^\times$. In this case $\pi(\mu |\,\,|^{1/2}, \mu |\,\,|^{-1/2})$ contains a unique irreducible sub-representation which is the twist $  \mathrm{St} \otimes \mu $ of the Steinberg representation, called a \emph{special representation}. The central character of $\mathrm{St} \otimes \mu$ equals $\mu^2$ and its conductor equals
    \[
\mathrm{cond}(\mathrm{St} \otimes \mu)=
\begin{cases}
2a(\mu) & \text{ if }\mu \text{ is\ ramified,}\\
1 & \text{ otherwise.}
\end{cases}
\]

	\item \smallskip 
	Any representation other than the above ones is of supercuspidal type.
\end{enumerate}

\medskip

\noindent \textbf{Elements of $\mathcal{B}_p$ and the correspondence:} On the other hand, an element of $\mathcal{B}_p$ is a $2$-dimensional Frobenius-semisimple representation of the Weil-Deligne group $WD(\Q_p)$ of $\Q_p$. This is a pair $(\phi, N)$ consisting of
\begin{enumerate} 
	\item [-]
	A $2$-dimensional semisimple representation $\phi : W(\Q_p) \rightarrow  {\rm GL}_2(\C)$ of the Weil group $W(\Q_p)$, which is continuous for the discrete topology in $\GL_2(\C)$;
	\item [-]
	$N$ is a nilpotent endomorphism of $\C^2$ such that $\phi(g)N\phi(g)^{-1}= p^{\omega_1(g)}N \quad \forall ~g \in W(\Q_p)$,
\end{enumerate} 
where $\omega_1$ is a character of $W(\Q_p)$ defined by $g|_{\overline{\F}_p}: x \mapsto x^{\omega_1(g)} \,\, \forall~ x \in \overline{\mathbb{F}}_p$, see \cite{TateCorvallis}.

As we are interested in computing the Galois orbits of a newform $f$, depending on the type of $\pi_p \in \mathcal{A}_p$ as described above, the corresponding element of $\mathcal{B}_p$ under local Langlands correspondence is given as follows:
\begin{enumerate} 
	\item \smallskip 
	If $\pi_p=\pi(\mu_1,\mu_2)$ is a principal series representation, the corresponding Weil–Deligne representation $(\phi, N)$ is given by
		$\phi(x)= \begin{bmatrix}
			\mu_1(x) & \\ & \mu_2(x)	\end{bmatrix} , \, x\in W(\Q_p) \text{ and } N=0.$
	\item \smallskip 
	If $\pi_p= \mathrm{St} \otimes \mu$ is of special type, then the Weil–Deligne representation attached to $\pi_p$ is a pair $(\phi, N)$ where
		$\phi(x)= \begin{bmatrix}
			\mu(x)|x|^{\frac{1}{2}} & \\ & \mu(x)|x|^{-\frac{1}{2}}	\end{bmatrix} , \, x\in W(\Q_p) \text{ and } N=\begin{bmatrix}
			0 & 1 \\ 0 & 0
		\end{bmatrix}.$
	\item \smallskip 
	
Let $\pi_p$ be a supercuspidal representation of $\mathrm{GL}_2(\Q_p)$. If $p$ is an odd prime, then $\pi_p$ is always of dihedral supercuspidal type (see \cite{Bump}). By the local Langlands correspondence, the Weil-Deligne representation $(\phi, N)$ associated to $\pi_p$ has $N=0$, and the Langlands parameter $\phi$ is of the form 
\begin{equation} \label{eq:ind_theta}
    \phi = \Ind_{W(K)}^{W(\Q_p)} \theta,
\end{equation}
where $K/\Q_p$ is a quadratic extension and $\theta$ is a character of $W(K)$ such that $\theta \neq \theta^\iota$. Here, $\iota \in W(\Q_p) \setminus W(K)$ is a representative of the non-trivial coset, and $\theta^\iota(x) = \theta(\iota x \iota^{-1})$ for $x \in W(K)$. With respect to a suitable basis, $\phi$ has the following matrix form:
\begin{equation} \label{eq:phi_matrix}
    \phi(x) = \begin{bmatrix} 
    \theta(x) & 0 \\ 
    0 & \theta^\iota(x) 
    \end{bmatrix} \quad \text{for } x \in W(K), \quad \text{and} \quad 
    \phi(\iota) = \begin{bmatrix} 
    0 & 1 \\ 
    \theta(\iota^2) & 0 
    \end{bmatrix}.
\end{equation}

Considering $\theta$ as a character of $K^{\times}$ via the isomorphism $W(K)^{ab}\cong K^{\times}$, $\phi$ is irreducible precisely when $\theta$ does not factor through the norm map $\text{Norm}: K^{\times} \rightarrow \Q_p^{\times}$. Let $\omega_{K / \Q_p}$ denote the quadratic character of $\Q_p$ associated by local class field theory to the extension $K/ \Q_p$. The central character of $\Ind^{W(\Q_p)}_{W(K)} \theta$ equals $\theta|_{\Q_p^{\times}}\cdot \omega_{K/\Q_p}$.

The Artin conductor $a(\phi)$ is related to the conductor of $\theta$, denoted $a(\theta)$, by the formula
\begin{equation*}
    a(\phi) = \dim(\theta) \, v(d_{K/\Q_p}) + f_{K/\Q_p} \, a(\theta),
\end{equation*}
where $d_{K/\Q_p}$ is the discriminant and $f_{K/\Q_p}$ is the residual degree of the extension $K/\Q_p$ (see \cite{Serre} for more details). Note that $\dim(\theta) = 1$. If $K/\Q_p$ is unramified, the discriminant is the unit ideal, meaning $v(d_{K/\Q_p}) = 0$, and the residual degree is $f_{K/\Q_p} = 2$. If $K/\Q_p$ is ramified and $p > 2$, the ramification is tame, which ensures $v(d_{K/\Q_p}) = 1$ and $f_{K/\Q_p} = 1$. Applying these facts, the conductor of $\pi_p$ (which equals $a(\phi)$) simplifies to
\begin{equation} \label{eq:conductor_Np}
    N_p = a(\pi_p) = 
    \begin{cases} 
        2a(\theta), & \text{if } K/\Q_p \text{ is unramified}, \\ 
        1+a(\theta), & \text{if } K/\Q_p \text{ is ramified}.
    \end{cases}
\end{equation}

For $p=2$, the classification is more complex because wild ramification allows for non-dihedral supercuspidal representations. In this case, the projective image of the Langlands parameter $\phi$ in $\mathrm{PGL}_2(\C)$ can also be isomorphic to $A_4$ (tetrahedral) or $S_4$ (octahedral).

\end{enumerate}
\section{Galois Orbits of Local Types}\label{localtypegaloisorbit}
Let $f(z) = \sum\limits_{n=1}^\infty a_n q^n$ be a modular form. For any continuous automorphism $\sigma \in \Gal(\C/\Q)$, the Galois conjugate of $f$ is defined by applying $\sigma$ to the Fourier coefficients of $f$, namely $f^\sigma(z)=\sum_{n=1}^\infty \sigma(a_n)q^n$.

\begin{prop}
    Let $f \in S_k(N, \Psi)$ where $\Psi$ is a quadratic character. Then the above Galois action is well-defined; that is, $f^\sigma \in S_k(N, \Psi)$.
\end{prop}

\begin{proof}
   The condition $f \in S_k(N, \Psi)$ implies $\langle d \rangle f = \Psi(d) f$. By \cite[Thm. 3.52]{Shimura}, the space $S_k(\Gamma_1(N))$ possesses a basis of forms with rational Fourier coefficients, ensuring that the Galois action stabilizes $S_k(\Gamma_1(N))$ and commutes with the diamond operators $\langle d \rangle$ for $d \in (\Z/N\Z)^\times$. Therefore, applying $\sigma$ to both sides yields $\langle d \rangle (f^\sigma) = \sigma(\Psi(d)) f^\sigma$. Because $\Psi$ is a quadratic character, we have $\sigma(\Psi(d)) = \Psi(d)$, implying $f^\sigma \in S_k(N, \Psi)$.
\end{proof}

A continuous automorphism $\sigma$ acts on the set of $2$-dimensional Frobenius-semisimple Weil-Deligne representations. Following \cite{dpt}, we make the following definition.

\begin{definition}
Given $\pi_1,\pi_2 \in \mathcal{A}_p$, we say they have \emph{Galois conjugate local inertial types} if there exists $\sigma \in \Gal(\C/\Q)$ such that the local inertial types of $\tau(\pi_1)$ and $\sigma(\tau(\pi_2))$ agree.
\end{definition}

This notion induces an equivalence relation on $\mathcal{A}_p$. We call the corresponding equivalence classes \emph{local type Galois orbits}. The same terminology will also be used for characters, viewed as $1$-dimensional automorphic forms.

\subsection{Counting local type Galois orbits with quadratic nebentypus}

Let $p$ be an odd prime, and let $f$ be a newform of level $N$ with a non-trivial nebentypus $\Psi$. We factor the level as $N=p^{N_p}N'$ with $p \nmid N'$ and the nebentypus as $\Psi=\Psi' \Psi_p$, where $\Psi':(\Z/N'\Z)^\times\to \C^\times$ is the prime-to-$p$ part and $\Psi_p:(\Z/p^{N_p}\Z)^\times \to \C$ is the $p$-part. To understand the local central character condition at $p$, we first adelize $\Psi$. For $x \in \Q_p^\times$, let $[x]=(1, \cdots, 1, x, 1, \cdots)$ be its corresponding element in $\mathbb{A}_\Q^\times$, where $x$ occupies the $p$-th place. Writing $x=p^ju$ for some $j \in \Z$ and $u \in \Z_p^\times$, the global character $\Psi$ induces a local character of $\Q_p^\times$ via the following formula:
\begin{equation} \label{adelic}
	\Psi([p^j u]) = \Psi'(p)^j \Psi_p^{-1}(u).
\end{equation}
By abuse of notation, we let $\Psi_p$ also denote this local character of $\Q_p^\times$. For the remainder of this section, we restrict our attention to quadratic $\Psi$ and newforms of level exactly $p^n$ (i.e., $N'=1$ and $N_p=n$). By Lemma \ref{L:quad-cond}, the local conductor $C_p$ of the quadratic character $\Psi_p$ satisfies  $C_p \leq 1$.

We define $\mathbf{LT}(p^n, \Psi)$ to be the number of distinct Galois orbits of local inertial types associated to the $p$-component of our newform $f$. More precisely, an inertial type is an isomorphism class of a two-dimensional representation of the inertia group $I_{\Q_p}$ at $p$. We count the orbits of these types under the action of the absolute Galois group $\mathrm{Gal}(\overline{\Q}/\Q)$, subject to the following local constraints dictated by the global setup:
\begin{enumerate}
    \item \textbf{Conductor Constraint:} The Artin conductor of the representation is exactly $n$.
    \item \textbf{Central Character Constraint:} By local class field theory, the central character of the representation restricts to the unit group $\Z_p^\times$ exactly as the inverse of the local nebentypus.
  
\end{enumerate} By the classification of elements of $\mathcal{A}_p$ and the local Langlands correspondence discussed in the last section, we compute this by summing the Galois orbits across the principal series, Steinberg, and supercuspidal types. 

Supercuspidal representations are induced from a character $\theta$ of a quadratic extension $K$ of $\Q_p$. For this induction to have the correct nebentypus, the central character constraint dictates, see \cite[Remark 4.5]{bmm}:
\begin{equation}\label{central-ch-condn}
    \theta|_{\Z_p^\times} = \Psi_p^{-1}\omega_{K/\Q_p},
\end{equation}
where $\omega_{K/\Q_p}$ is the quadratic character associated to the quadratic extension $K/\Q_p$ by local class field theory. To count the Galois orbits of these types, we quote the following result directly from \cite[Lemma 2.9]{dpt}:

\begin{lemma}
    Let $\sigma \in \Gal(\C/\Q)$ and suppose $(K, \theta)$ parametrizes a supercuspidal representation (where $\theta: W(K) \rightarrow \C^{\times}$). Then $\sigma (K, \theta)=(K, \sigma \circ \theta)$.
\end{lemma}

In particular, two non-isomorphic supercuspidal representations have Galois-conjugate inertial types if and only if they are induced from the same quadratic extension $K$ and their inducing characters are Galois conjugate; equivalently, one is obtained from the other by raising to a power coprime to its order.

Theorem~\ref{units-in-residue} gives explicit descriptions of the unit groups $(\mco_K/\mathfrak{p}^n)^\times$. In what follows, $\sigma_0(a)$ denotes the number of positive divisors of an integer $a$.

\begin{thm} \label{thm:orbits_primitive}
Let $p\ge 3$ be an odd prime, $K=\Q_p(\sqrt{d})$ be a quadratic extension of $\Q_p$, and $\Psi_p$ be a quadratic character of $\Q_p^\times$. Let $m$ be the largest odd factor of $p+1$. Let $\theta: K^\times \to \C^\times$ be a primitive character of conductor $n$ satisfying the central character condition $\theta|_{\Z_p^\times} = \Psi_p^{-1}\omega_{K/\Q_p}$. The number of local type Galois orbits of such primitive characters $\theta$ is given in Table \ref{Table:character_orbit_count_ramified} when $\Psi_p$ is tamely ramified, and in Table \ref{Table:character_orbit_count_unram} when $\Psi_p$ is unramified.

\begin{table}[h]
\centering
\scalebox{0.95}{
\begin{tabular}{||c|c|c|c||c||}
  \hline
  $K$ & $e$ & $p$ & $n$ & \# Orbits \\
  \hline\hline
  --- & $1$ & $\neq 2$ & $\ge 1$ & $\sigma_0(m)$ \\
  \hline
  $\neq \Q_3(\sqrt{-3})$ & $2$ & $\neq 2$ & $1$ & $0$ \\
  \hline
  $\neq \Q_3(\sqrt{-3})$ & $2$ & $\neq 2$ & $\stackrel{n \ge2}{\text{odd} \, \vert\, \text{even}}$ & $0 \, \vert \, 1$ \\
  \hline
  $\Q_3(\sqrt{-3})$ & $2$ & $3$ & $1 \, \vert \, 2$ & $0 \, \vert \, 1$ \\
  \hline
  $\Q_3(\sqrt{-3})$ & $2$ & $3$ & $\stackrel{n\ge 3}{\text{odd} \, \vert \, \text{even}}$ & $0 \, \vert \, 3$ \\
  \hline
\end{tabular}
}
\vspace{2mm}
\caption{Number of local type Galois orbits of primitive characters for odd primes with a \textbf{tamely ramified} quadratic nebentypus $\Psi_p$.}
\label{Table:character_orbit_count_ramified}
\end{table}

\begin{table}[h]
\centering
\scalebox{0.95}{
\begin{tabular}{||c|c|c|c||c||}
  \hline
  $K$ & $e$ & $p$ & $n$ & \# Orbits \\
  \hline\hline
  --- & $1$ & $\neq 2$ & $1$ & $\sigma_0(p+1)-1$ \\
  \hline
  --- & $1$ & $\neq 2$ & $\ge 2$ & $\sigma_0(p+1)$ \\
  \hline
  $\neq \Q_3(\sqrt{-3})$ & $2$ & $\neq 2$ & $1$ & $1$ \\
  \hline
  $\neq \Q_3(\sqrt{-3})$ & $2$ & $\neq 2$ & $\stackrel{n \ge2}{\text{odd} \, \vert\, \text{even}}$ & $0 \, \vert \, 1$ \\
  \hline
  $\Q_3(\sqrt{-3})$ & $2$ & $3$ & $1 \, \vert \, 2$ & $1 \, \vert \, 1$ \\
  \hline
  $\Q_3(\sqrt{-3})$ & $2$ & $3$ & $\stackrel{n\ge 3}{\text{odd} \, \vert \, \text{even}}$ & $0 \, \vert \, 3$ \\
  \hline
\end{tabular}
}
\vspace{2mm}
\caption{Number of local type Galois orbits of primitive characters for odd primes with an \textbf{unramified} quadratic nebentypus $\Psi_p$.}
\label{Table:character_orbit_count_unram}
\end{table}
\end{thm}

\begin{proof}
Using the group structure and generators given in Table \ref{Table:units-in-residue}, we evaluate the possible character values at each generator. We distinguish the cases according to the ramification of $K/\Q_p$. When $\Psi_p$ is unramified, the result follows from the statement and proof of \cite[Theorem 2.11]{dpt}. We only note that in the situation $n=1$, with $K/\Q_p$ unramified and $\Psi_p$ unramified, excluding the trivial character reduces the count to $\sigma_0(p+1)-1$ orbits, rather than $\sigma_0(p+1)$ as stated in loc.\ cit. Therefore, in what follows, we focus on the case where $\Psi_p$ is tamely ramified.

\smallskip
\noindent
\underline{$K/\Q_p$ unramified:}
Here $\omega_{K/\Q_p}$ is unramified, so $\theta|_{\Z_p^\times}=\Psi_p^{-1}$. Since $\xi_{p-1}\in\Z_p^\times$, we have
\[
\theta(\xi_{p^2-1})^{p+1}=\Psi_p^{-1}(\xi_{p-1}).
\]

As $\Psi_p$ is tamely ramified, we have $\Psi_p^{-1}(\xi_{p-1})=-1$. Thus $\theta(\xi_{p^2-1})^{p+1}=-1$, implying its order divides $2(p+1)$ but not $p+1$. Writing $p+1=2^k m$ with $m$ odd, the possible orders are $2^{k+1}d$ with $d\mid m$, giving $\sigma_0(m)$ orbits for all $n$. As $1+p \in \Z_p^\times$, character values on this generator are fixed by the central character condition. Wild generators only control the conductor and do not affect the orbit count, so this holds for all $n\ge1$.

\smallskip
\noindent
\underline{$K/\Q_p$ ramified:}
The central character condition fixes $\theta$ on $\Z_p^\times$. Let $\varpi$ denote $\sqrt{d}$, the uniformizer of $\Q_p(\sqrt{d})$.

\emph{Case $n=1$.}
Then $\theta$ factors through $\F_p^\times=\langle \xi_{p-1}\rangle$. As $\Psi_p$ is tamely ramified, the central character condition forces $\theta(\xi_{p-1})=1$, so $\theta$ is trivial on $\mco_K^\times$, contradicting primitivity; hence no orbits occur.

\emph{Case $n \ge 2$ odd.}

Write $n=2k+1$. Then primitivity requires $\theta$ to be nontrivial on the quotient $U_K^{(2k)}/U_K^{(2k+1)}$, which is generated by $1+p^k$. However, $1+p^k\in 1+p\Z_p$, and the central character condition forces $\theta$ to be trivial on all such elements. This contradiction shows that no primitive characters exist.

\emph{Case $n=2k\ge2$ even.}
If $K\neq \Q_3(\sqrt{-3})$, the values on $\xi_{p-1}$ and $1+p$ are fixed to $1$ by the central character condition on $\Z_p^\times$, whereas the value $\theta(1+\varpi)$ can be any of the $\phi(p^k)$ primitive $p^k$-th roots of unity. As roots of unity with the same order are Galois conjugate, all such characters merge into exactly $1$ Galois orbit for any $n \ge 2$.

If $K=\Q_3(\sqrt{-3})$, the value of $\theta(4)$ is fixed by the central character condition. There is an additional torsion generator $\xi_3$ in $\mco_K^\times$. For $n=2$, the subgroup generated by $1+\varpi$ is trivial, so primitivity forces $\theta(\xi_3)\neq 1$, giving two possibilities which form a single Galois orbit. For $n\ge4$, the value $\theta(1+\varpi)$ must be a primitive $3^{k-1}$-th root of unity, giving $\phi(3^{k-1})$ choices, and $\theta(\xi_3)$ contributes $3$ independent choices. Hence there are $3\phi(3^{k-1})$ primitive characters in total. Fixing character value on $\zeta_3$ forces the character value on the generator $1+\sqrt{-3}$. Consequently, we cannot treat the orbits of these two generators independently. Because $k \ge 2$, the image of $\theta(\zeta_3)$ lies entirely within the cyclotomic field $\Q(\xi_{3^{k-1}})$ generated by $\theta(1+\varpi)$, and therefore the action of the absolute Galois group on $\theta$ factors through $\Gal(\Q(\xi_{3^{k-1}})/\Q)$. Since any automorphism stabilizing $\theta$ must fix the primitive $3^{k-1}$-th root of unity $\theta(1+\varpi)$, this action is free. Therefore, each orbit has size $\phi(3^{k-1})$, yielding a total of $\frac{3\phi(3^{k-1})}{\phi(3^{k-1})}=3$ orbits.
\end{proof}

Let $K/\Q_p$ be a quadratic extension, and let $\theta$ be a character of the Weil group of $K$. Recall that when $p$ is odd, every supercuspidal representation arises, via the local Langlands correspondence, from a Weil-Deligne representation with trivial monodromy that is obtained by inducing $\theta$ from the Weil group of $K$ to that of $\Q_p$. Moreover, this induced representation is irreducible precisely when $\theta$ does not factor through the norm map from $K^\times$ to $\Q_p^\times$.

\begin{lemma} \label{inertia_grand_lemma_complete}
    Let $p$ be an odd prime, $K/\Q_p$ a quadratic extension, and $\Psi_p$ a quadratic character of $\Q_p^\times$. Let $\theta: K^\times \to \C^\times$ be a character such that $\theta|_{\Z_p^\times} = \omega_{K/\Q_p} \Psi_p^{-1}$. If $\theta$ factors through the norm map, the local type Galois orbits of $\theta$ are exactly given by:
    
    \begin{itemize}
        \item \textbf{When $\Psi_p$ is unramified:}
        \begin{itemize}
            \item $K/\Q_p$ unramified: exactly $1$ orbit of conductor $0$, and $1$ orbit of conductor $1$.
            \item $K/\Q_p$ ramified: exactly $1$ orbit of conductor $1$ if $p \equiv 1 \pmod 4$, and $0$ orbits if $p \equiv 3 \pmod 4$.
        \end{itemize}
        \item \textbf{When $\Psi_p$ is tamely ramified:}
        \begin{itemize}
            \item $K/\Q_p$ unramified: exactly $1$ orbit of conductor $1$ if $p \equiv 1 \pmod 4$, and $0$ orbits if $p \equiv 3 \pmod 4$.
            \item $K/\Q_p$ ramified: exactly $1$ orbit of conductor $0$.
        \end{itemize}
    \end{itemize}
\end{lemma}

\begin{proof}
Write $\theta=\varphi\circ N_{K/\Q_p}$. The central character condition gives
\[
\varphi(x)^2=\omega_{K/\Q_p}(x)\Psi_p(x)\quad \text{for }x\in\Z_p^\times.
\]
Since the right-hand side is trivial on $1+p\Z_p$ and $p\neq2$, by Lemma \ref{Hensel} it follows that $\varphi$ is trivial on $1+p\Z_p$, hence factors through $\F_p^\times$. As before, the unramified nebentypus case follows from the statement and proof of \cite[Lemma 2.12]{dpt}. So we only consider the case where $\Psi_p$ is tamely ramified.

\smallskip
\noindent
\textbf{Case 1: $K/\Q_p$ unramified.}
Here $N_{K/\Q_p}(\mco_K^\times)=\Z_p^\times$. Because of this surjectivity, $\theta$ and $\varphi$ uniquely determine one another. As $\Psi_p$ is tamely ramified, the equation $\varphi^2=\Psi_p$ has solutions if and only if $p \equiv 1 \pmod 4$, in which case the two solutions are conjugate, giving one orbit of conductor $1$; otherwise, there are none.

\smallskip
\noindent
\textbf{Case 2: $K/\Q_p$ ramified.}
Here $N_{K/\Q_p}(\mco_K^\times) = (\Z_p^\times)^2$. For $y \in \mco_K^\times$, write $N_{K/\Q_p}(y)=a^2$ with $a \in \Z_p^\times$. Then
\[
\theta(y) = \varphi(a^2) = \varphi^2(a) = \omega_{K/\Q_p}(a)\Psi_p(a),
\]
so $\theta$ is independent of the choice of $\varphi$. To ensure that $\theta$ is well-defined, the above expression must be independent of the choice of $a$, that is, invariant under $a \mapsto -a$. This requires
\[
\omega_{K/\Q_p}(-1)\Psi_p(-1)=1.
\]

As $\Psi_p$ is tamely ramified, then $\omega_{K/\Q_p}=\Psi_p$ on $\Z_p^\times$, so the condition is satisfied. In this case $\theta$ is trivial on $\mco_K^\times$, giving a unique orbit of conductor $0$, which is the orbit of the trivial character.
\end{proof}

\begin{thm}\label{thm:LTformula_nebentypus}
  Let $p \neq 2$ be a prime number. Let $m$ be the largest odd divisor of $p+1$. The values of $\mathbf{LT}(p^n, \Psi)$ for newforms with a non-trivial quadratic nebentypus $\Psi$ are given in Table \ref{table:nebentypus_ramified} when the local nebentypus $\Psi_p$ is tamely ramified, and in Table \ref{table:nebentypus_unramified} when $\Psi_p$ is unramified. 

\begin{table}[h]
\centering
\scalebox{0.95}{
\begin{tabular}{||l|c|c|c|c||}
\hline
$n$ & P.S. & St & S.C.U. & S.C.R. \\
\hline\hline
1 & 1 & --- & --- & --- \\
\hline
2 & $\sigma_0\left(\frac{p-1}{2}\right)-1$ & 
$\begin{cases} 1 & p \equiv 1 \pmod{4} \\ 0 & p \equiv 3 \pmod{4} \end{cases}$ & 
$\begin{cases} \sigma_0(m)-1 & p \equiv 1 \pmod{4} \\ \sigma_0(m) & p \equiv 3 \pmod{4} \end{cases}$ 
& --- \\

\hline
$\substack{p \neq 3 \\ n \ge 3 \text{ odd}}$ & --- & --- & --- & 2 \\
\hline
$\substack{p = 3 \\ n = 3 }$ & --- & --- & --- & 2 \\
\hline
$\substack{p = 3 \\ n \ge 5 \text{ odd}}$ & --- & --- & --- & 4 \\
\hline   
$n \ge 4 \text{ even}$ & $\sigma_0\left(\frac{p-1}{2}\right)$ & --- & $\sigma_0(m)$ & --- \\
\hline
\end{tabular}
}
\vspace{2mm}
\caption{Values for $\mathbf{LT}(p^n, \Psi)$ with a \textbf{tamely ramified} quadratic nebentypus $\Psi_p$.}
\label{table:nebentypus_ramified}
\end{table}

\begin{table}[h]
\centering
\scalebox{0.95}{
\begin{tabular}{||l|c|c|c|c||}
\hline
$n$ & P.S. & St & S.C.U. & S.C.R. \\
\hline\hline
1 & --- & 1 & --- & --- \\
\hline
2 & $\sigma_0(p-1)-1$ & 1 & $\sigma_0(p+1)-2$ & --- \\
\hline
$\substack{p \neq 3 \\ n \ge 3 \text{ odd}}$ & --- & --- & --- & 2 \\
\hline
$\substack{p = 3 \\ n = 3 }$ & --- & --- & --- & 2 \\
\hline
$\substack{p = 3 \\ n \ge 5 \text{ odd}}$ & --- & --- & --- & 4 \\
\hline   
$n \ge 4 \text{ even}$ & $\sigma_0(p-1)$ & --- & $\sigma_0(p+1)$ & --- \\
\hline
\end{tabular}
}
\vspace{2mm}
\caption{Values for $\mathbf{LT}(p^n, \Psi)$ with an \textbf{unramified} quadratic nebentypus $\Psi_p$.}
\label{table:nebentypus_unramified}
\end{table}
\end{thm}

\begin{proof}
We compute $\mathbf{LT}(p^n,\Psi)$ by summing the valid Galois orbits across the principal series, Steinberg, and supercuspidal local types.

\noindent $\bullet$ {\bf Principal Series:} 
The local representation is of the form $\pi(\chi_1,\chi_2)$, with the central character condition $\chi_1\chi_2|_{\Z_p^\times}=\Psi_p^{-1}|_{\Z_p^\times}$. By Lemma \ref{L:quad-cond}, $\Psi_p$ has conductor at most $1$. The conductor is $n=a(\chi_1)+a(\chi_2)$.

For $n=1$, one has $\{a(\chi_1),a(\chi_2)\}=\{1,0\}$. Assume $\chi_2$ is unramified. Then $a(\chi_2)=0$ and $\chi_1|_{\Z_p^\times}=\Psi_p^{-1}|_{\Z_p^\times}$. If $\Psi_p$ is unramified, this forces $a(\chi_1)=0$, a contradiction, so there are no orbits. If $\Psi_p$ is tamely ramified, the inertial restrictions are uniquely $\{\Psi_p^{-1},1\}$, giving exactly one Galois orbit.

For $n\ge2$, the condition $a(\chi_1\chi_2)\le1$ forces $a(\chi_1)=a(\chi_2)$, since otherwise $\max(a(\chi_1),a(\chi_2))\le1$ and hence $n=1$. Thus $n=2a(\chi_1)$ is even; write $d=n/2$. The restriction of $\chi_1$ to inertia is a primitive character of $(\Z/p^d\Z)^\times$, whose wild part has exact order $p^{d-1}$, so Galois orbits are determined by the tame component. The central character condition gives $\chi_2=\chi_1^{-1}\Psi_p^{-1}$, so we count unordered pairs $\{\chi_1,\chi_1^{-1}\Psi_p^{-1}\}$ up to conjugacy.

\begin{itemize}
\item {\bf Unramified $\Psi_p$:} This follows from \cite[Theorem 2.7]{dpt}.
\item {\bf Tamely ramified $\Psi_p$:} Here $\Psi_p|_{\Z_p^\times}$ is quadratic. The pairs are $\{\chi_1,\chi_1^{-1}\Psi_p^{-1}\}$. Passing to the quotient of the tame character group by $\langle \Psi_p\rangle$ (of order $(p-1)/2$), these descend to pairs $\{\overline{\chi}_1,\overline{\chi}_1^{-1}\}$. The two lifts are exchanged by inversion and hence lie in the same Galois orbit, so orbits correspond to cyclic subgroups of the quotient. This gives $\sigma_0\!\left(\frac{p-1}{2}\right)$ orbits for $d>1$, and $\sigma_0\!\left(\frac{p-1}{2}\right)-1$ for $d=1$.
\end{itemize}

\smallskip
\noindent
\textbf{Steinberg.}
Let $\pi_p = \mathrm{St} \otimes \mu$. The central character condition requires $\mu^2 |_{\Z_p^\times} = \Psi_p^{-1}|_{\Z_p^\times}$. The conductor of this representation is $n = 1$ if the twist $\mu$ is unramified ($a(\mu) = 0$), and $n = 2a(\mu)$ if $\mu$ is ramified. 

Because $p \neq 2$, the squaring map on characters preserves the conductor for any ramified character. Since $\Psi_p^{-1}$ has conductor at most $1$, any valid twist $\mu$ must satisfy $a(\mu) \le 1$. Consequently, Steinberg types only contribute at $n=1$ and $n=2$.
\begin{itemize}
    \item If $\Psi_p$ is \textbf{unramified}: This follows from \cite[Theorem 2.7]{dpt}.
    \item If $\Psi_p$ is \textbf{tamely ramified}: The condition requires $\mu^2 |_{\Z_p^\times}$ to be the unique tamely ramified quadratic character. This implies $\mu |_{\Z_p^\times}$ must be a character of exact order $4$, forcing $a(\mu) = 1$ (which gives $n=2$). Such a character exists if and only if the cyclic group $(\Z/p\Z)^\times$ has order divisible by $4$, meaning $p \equiv 1 \pmod 4$. The two solutions are inverse to each other and thus Galois conjugate. This yields exactly $1$ Galois orbit at $n=2$ if $p \equiv 1 \pmod 4$, and $0$ orbits if $p \equiv 3 \pmod 4$.
\end{itemize}

\smallskip
\noindent
\textbf{Supercuspidal.}
These correspond to $\Ind_{W(K)}^{W(\Q_p)}\theta$, where $K/\Q_p$ is quadratic and $\theta$ does not factor through the norm map. The central character condition is $\theta|_{\Z_p^\times}=\Psi_p^{-1}\omega_{K/\Q_p}$.

\smallskip
\noindent
$\ast$ \textit{Unramified extensions (S.C.U.).}
The conductor of the representation is given by $n = 2a(\theta)$, ensuring $n$ is always even.
When $\Psi_p$ is tamely ramified, the number of primitive character orbits is $\sigma_0(m)$ for all $n \ge 2$ by Theorem \ref{thm:orbits_primitive}. At $n=2$, exactly one of these orbits factors through the norm if and only if $p \equiv 1 \pmod 4$ by Lemma \ref{inertia_grand_lemma_complete}. Subtracting this yields the counts in Table \ref{table:nebentypus_ramified}. For $n \ge 4$, no orbits factor through the norm by Lemma \ref{inertia_grand_lemma_complete}, leaving $\sigma_0(m)$.
When $\Psi_p$ is unramified, the number of primitive character orbits is $\sigma_0(p+1)-1$ for $n=2$, and $\sigma_0(p+1)$ for $n \ge 4$ by Theorem \ref{thm:orbits_primitive}. At $n=2$, exactly one orbit factors through the norm regardless of $p \pmod 4$ by Lemma \ref{inertia_grand_lemma_complete}. Subtracting this yields exactly $\sigma_0(p+1)-2$ valid orbits.

\smallskip
\noindent
$\ast$ \textit{Ramified extensions (S.C.R.).}
The conductor is given by $n = a(\theta) + 1$. Because valid primitive characters only exist when $a(\theta)$ is even by Theorem \ref{thm:orbits_primitive}, the representation conductor $n$ must be odd, with the exception of $a(\theta)=1$ ($n=2$) for unramified $\Psi_p$ cases.

When $\Psi_p$ is tamely ramified, no primitive characters exist for $a(\theta) \le 1$, meaning no orbits exist for $n \le 2$ (by Theorem \ref{thm:orbits_primitive}). For odd $n \ge 3$, the character conductor is $a(\theta) = n-1 \ge 2$, and no characters factor through the norm map by Lemma \ref{inertia_grand_lemma_complete}. At $a(\theta)=2$ ($n=3$), Theorem \ref{thm:orbits_primitive} dictates there is exactly $1$ primitive orbit for any $K$. Summing over the two ramified quadratic extensions yields $1+1=2$ orbits for all odd primes $p$ at $n=3$. For odd $n \ge 5$, we have $a(\theta) \ge 4$. Here, each extension with $K \neq \Q_3(\sqrt{-3})$ contributes exactly $1$ orbit. Summing over the two extensions gives $1+1=2$ orbits for $p \neq 3$. For $p=3$, the anomalous extension $\Q_3(\sqrt{-3})$ contributes $3$ orbits, elevating the total count to $1+3=4$ for $n \ge 5$.

When $\Psi_p$ is unramified, the counts differ at $a(\theta)=1$ ($n=2$). Here each extension contributes exactly $1$ orbit by Theorem \ref{thm:orbits_primitive}. These two orbits factor through the norm if $p \equiv 1 \pmod 4$ (yielding $0$ valid orbits). While both survive the norm map condition if $p \equiv 3 \pmod 4$, their resulting local inertial types are known to be isomorphic to unramified supercuspidal types (cf.\ \cite[2.7]{gerardinkutzko}). Because $\mathbf{LT}$ enumerates distinct local types, these orbits are already accounted for in the S.C.U.\ column. To avoid double-counting, we omit them from the S.C.R.\ column, yielding $0$ distinct orbits at $n=2$.
\end{proof}

\begin{remark}
    We note a minor correction to the literature regarding the base case of trivial nebentypus for $p=3$. In \cite[Theorem 2.7]{dpt}, it is asserted that there are exactly four Galois orbits of local inertial types for $p=3$ and odd conductor exponent $n \ge 3$. However, as pointed out in \cite[Remark 3.8]{KM}, this cannot hold for $n=3$, as it would require the existence of eight supercuspidal representations of conductor $3^3$, contradicting the known total of four. The correct number of local types for $n=3$ is two. Our tables reflect this bifurcation for $p=3$.
\end{remark}

\medskip

\section{Global Galois Invariants and Atkin-Li Pseudo-eigenvalues}
\label{sec: types-of-modular-forms}

Let $f =\sum a_n q^n \in S_k(N,\Psi)$ be a newform. Its Hecke field is denoted by $E_f=\Q(\{a_n\})$, obtained by attaching all Fourier coefficients of $f$ to $\Q$. For a prime $\ell$, let $E_{f, \ell}$ denote the completion of $E_f$ at $\ell$, $\mathcal{O}_\ell$ be the ring of integers in $E_{f, \ell}$, $\pi$ be a uniformizer in $\mathcal{O}_\ell$, and $\mathbb{F}_\ell = \mathcal{O}_\ell/\pi$ be the residue field. Let $\chi_\ell$ denote the $\ell$-adic cyclotomic character, and $\Frob_p$ denote an arithmetic Frobenius element at $p$. Deligne \cite{Deligne} attached a continuous irreducible $\ell$-adic Galois representation
$$ \rho_f=\rho_{f,\ell}: \Gal(\overline{\Q}/\Q) \to \GL_2(\mathcal{O}_\ell) $$ 
to $f$ such that $\rho_f$ is unramified outside $\ell N$ and 
\begin{equation} \label{Galois}
    \det(\rho_{f})=\chi_\ell^{k-1}\Psi, ~~ \mathrm{trace}(\rho_{f}(\Frob_p))=a_f(p) \text{ for } p \nmid \ell N.
\end{equation}

\begin{lemma} \label{p-part}
    Modular forms with the same local inertial type at a prime $p$ have the same $p$-part of their nebentypus.
\end{lemma}
\begin{proof}
By the local Artin isomorphism, the inertia group at $p$ corresponds to $\Z_p^\times$. Hence, from Equation \ref{adelic}, the restriction of $\Psi$ to the inertia is just the inverse of $\Psi_p$. Also, the character $\chi_\ell$ is unramified at $p$. Therefore Equation \ref{Galois} implies $\det(\rho_{f}|_{I_p})=\Psi_p^{-1}$ and we conclude the result.
\end{proof}

\begin{lemma}\label{prin}

Let $f \in S_k(N,\Psi)$ be a newform with nebentypus $\Psi$ of exact order $m$. For a prime $p \mid N$, suppose $\pi_{f,p}$ is a principal series $\pi(\chi_1,\chi_2)$. Let $d$ be the order of $\chi_1|_{\mathbb{Z}_p^\times}$. Then $E_f$ contains the field $\mathbb{Q}(\xi_d + \zeta \xi_d^{-1})$ for some $m$-th root of unity $\zeta \in \mu_m$.

\end{lemma}

\begin{proof}
For any prime $\ell \neq p$, the restriction of $\rho_{f,\ell}$ to inertia satisfies $\rho_{f,\ell}|_{I_p} \cong \chi_1 \oplus \chi_2$. Here, we identify the inertia group $I_p$ with the local unit group $\mathbb{Z}_p^\times$ via the local Artin reciprocity map, denoted $\mathrm{Art}_{\mathbb{Q}_p} : \mathbb{Q}_p^\times \to W_{\mathbb{Q}_p}^{\mathrm{ab}}$. The central character condition gives $\chi_1(x)\chi_2(x)=\Psi_p^{-1}(x)$ for $x \in \mathbb{Z}_p^\times$. Since $\Psi_p$ takes values in $\mu_m$, $\Psi_p^{-1}(x) = \zeta$ for some $m$-th root of unity $\zeta \in \mu_m$, hence $\chi_2(x)= \zeta \chi_1(x)^{-1}$. Therefore, evaluating the trace on an inertia element corresponding to $x \in \mathbb{Z}_p^\times$, we have:$$\tr(\rho_{f,\ell}(\mathrm{Art}_{\mathbb{Q}_p}(x)))=\chi_1(x) + \zeta \chi_1(x)^{-1}.$$ As $\chi_1|_{\mathbb{Z}_p^\times}$ has order $d$, for a suitable choice of $x$, these traces generate the field $F_\chi=\mathbb{Q}(\xi_d + \zeta \xi_d^{-1})$.

From here, we adapt the argument of \cite[Lemma 3.1]{dpt} to show that $F_\chi \subset E_f$.  Since $F_\chi/\mathbb{Q}$ is an abelian extension, Chebotarev's density theorem guarantees the existence of infinitely many primes inert in $F_\chi$. Choose such a prime $\ell \neq p$ for $\rho_{f,\ell}$. Suppose $F_\chi \not\subset E_f$. Then $F_\chi$ is a nontrivial extension of $F_\chi \cap E_f$. As $\ell$ is inert in $F_\chi$, the induced extension of local fields is nontrivial, yielding$[(F_\chi)_\ell : E_{f,\ell} \cap (F_\chi)_\ell] > 1$. However, the trace $\chi_1(x) + \zeta \chi_1(x)^{-1} = \mathrm{tr}(\rho_{f,\ell}(\text{Art}_{\mathbb{Q}_p}(x)))$ belongs to $E_{f,\ell}$ for all $x \in \mathbb{Z}_p^\times$. Consequently, the generators of $F_\chi$ lie in $E_{f,\ell}$, forcing $(F_\chi)_\ell \subseteq E_{f,\ell}$, which is a contradiction to the above inequality of extension degree. Hence $F_\chi \subset E_f$.
\end{proof}

\begin{lemma} \label{sup}
Let $f \in S_k(N,\Psi)$ be a newform with nebentypus $\Psi$ of exact order $m$. For a prime $p \mid N$, suppose $\pi_{f,p}$ is a non-sporadic supercuspidal representation, say $\pi_{f,p} \cong \Ind_{W(K)}^{W(\Q_p)} \theta$ for a quadratic extension $K/\Q_p$. Let $d$ be the order of $\theta|_{\mco_K^{\times}}$, where $\mco_K$ is the ring of integers of $K$. Then $E_f$ contains the field $\mathbb{Q}(\xi_d + \zeta \xi_d^{-1})$ for some $m$-th root of unity $\zeta \in \mu_m$.
\end{lemma}

\begin{proof}
As $\mathbb{C}^\times$ is abelian, the character $\theta$ factors uniquely through the abelianization $W(K)^{ab} = W(K)/[W(K), W(K)]$. The Artin reciprocity gives an isomorphism $\text{Art}_K : K^\times \xrightarrow{\sim} W(K)^{ab}$. We identify the inertia subgroup of $K$ with $\mathcal{O}_K^\times$ under this map. Thus, for any $x \in \mathcal{O}_K^\times$, we evaluate $\theta(x)$ as $\theta(\text{Art}_K(x))$. The restriction of the local $\ell$-adic Galois representation to $W(K)$ decomposes as $\rho_{f, \ell}|_{W(K)} \cong \theta \oplus \theta^\iota$, where $\iota$ is a fixed element in $W(\mathbb{Q}_p) \setminus W(K)$ lifting the non-trivial involution of $\text{Gal}(K/\mathbb{Q}_p)$. The conjugate character is defined by $\theta^\iota(g)=\theta(\iota^{-1} g \iota)$ for $g \in W(K)$.  

By the Galois equivariance of the local Artin map (see condition (W2) in Sec. 1 of \cite{TateCorvallis}), conjugation by $\iota$ on the Weil group intertwines with the Galois action on the field $K^\times$. Specifically:
$$\iota^{-1} \text{Art}_K(x) \iota \equiv \text{Art}_K(\overline{x}) \pmod{[W(K), W(K)]},$$
where $\overline{x}$ is the Galois conjugate of $x$ under the non-trivial involution in $\text{Gal}(K/\mathbb{Q}_p)$. 

Consequently, $\iota^{-1} \text{Art}_K(x) \iota$ and $\text{Art}_K(\overline{x})$ represent the same class in $W(K)^{ab}$. Since $\theta$ factors through this abelianization, we have
$$\theta^\iota(x) = \theta(\iota^{-1} \text{Art}_K(x) \iota) = \theta(\text{Art}_K(\overline{x})) = \theta(\overline{x}).$$

Evaluating the trace of the Galois representation on an inertia element $x \in \mathcal{O}_K^\times$, we obtain
$$\mathrm{tr}(\rho_{f,\ell}(\text{Art}_K(x))) = \theta(x) + \theta^\iota(x) = \theta(x) + \theta(\overline{x}).$$

To evaluate $\theta(\overline{x})$, observe that the norm $x\overline{x} = N_{K/\mathbb{Q}_p}(x)$ lies in $\mathbb{Z}_p^\times$. The central character condition requires $\theta|_{\mathbb{Z}_p^\times} = \Psi_p^{-1}\omega_{K/\mathbb{Q}_p}$. Because $\omega_{K/\mathbb{Q}_p}$ is trivial on norms and $\Psi_p$ has order dividing $m$, the right-hand side evaluates to some $m$-th root of unity $\zeta \in \mu_m$. This forces
$$\theta(x)\theta(\overline{x}) = \theta(x\overline{x}) = \zeta,$$
which implies $\theta(\overline{x}) = \zeta \theta(x)^{-1}$.

Substituting this back into the trace yields:
$$\mathrm{tr}(\rho_{f,\ell}(\text{Art}_K(x))) = \theta(x) + \zeta \theta(x)^{-1}.$$

As $\theta|_{\mathcal{O}_K^\times}$ has order $d$, the value $\theta(x)$ is a primitive $d$-th root of unity $\xi_d$ for a suitable choice of $x$. Therefore, these traces generate the field $F_\theta = \mathbb{Q}(\xi_d + \zeta \xi_d^{-1})$. From this point, the argument used in Lemma \ref{prin} applies verbatim, allowing us to conclude that $F_\theta \subset E_f$. 
\end{proof}

\begin{thm}\label{orbit-preserves-type}
Let $f \in S_k(N, \Psi)$ be a newform with a quadratic nebentypus $\Psi$. For a prime $p \mid N$, the set $\{\widetilde{\pi}_{\sigma(f),p} : \sigma \in \Gal(\overline{\mathbb{Q}}/\mathbb{Q})\}$ of local types at $p$ of the Galois conjugates of $f$ equals the local type Galois orbit of $\widetilde{\pi}_{f,p}$.
\end{thm}

\begin{proof}
   The proof of this theorem is analogous to that of \cite[Theorem 3.3]{dpt}. Lemmas \ref{prin} and \ref{sup} ensure that each local type in the Galois orbit of $\widetilde{\pi}_{f,p}$  is realized by a global conjugate of $f$.
\end{proof}


\subsection{Pseudo-eigenvalues}

When studying Galois orbits of modular forms, there is another natural invariant to consider, namely the Atkin-Li pseudo-eigenvalues at each prime \(p\) dividing the level. Let \(f \in S_k(N,\Psi)\) be a newform. Following \cite{al}, for each \(p \mid N\) with $\val_p(N)=N_p \geq 1$, there is an Atkin-Li operator \(W_{p}\) acting on \(S_k(N,\Psi)\)  which commutes with the Hecke operators up to some nebentypus twist.

If \(\lambda_p(f)\) denotes the pseudo-eigenvalue of \(W_p\), then
\[
W_p f = \lambda_p(f)\,g,
\]
for some newform $g$. By \cite[Theorem~1.1]{al}, the pseudo-eigenvalue \(\lambda_p(f)\) is an algebraic number of absolute value \(1\).

Consider the adelized holomorphic modular form of $\GL_2(\mathbb{A}_\Q)$ attached to $f$ which will again be denoted by $f$. Let $\mathscr{F}f$ denote the Whittaker function for a modular form $f$ and it is defined on $\GL_2(\mathbb{A}_\Q)$ by
\[\mathscr{F}f (G) = \int_{\Q \backslash \mathbb{A}} f\left( \begin{bmatrix} 1 & u \\ 0 & 1  \end{bmatrix}G \right) \psi(-u) \, du,\]
where $\psi$ is the unitary character on $\Q \backslash \mathbb{A}$ such that $\psi_\infty(x)=e^{2\pi i x}$. The relation between the Whittaker function of $W_pf$ evaluated at identity and the local epsilon factor at $p$ is given by the theorem below \cite[Theorem]{flath}.
\begin{thm} \label{pseudo}
	$\mathscr{F} W_p f (I) ={p^{-N_p/2}} \Psi_p(-\frac{1}{p^{N_p}}) \varepsilon_p(\pi_p, \psi_p)$,
	where $\pi_p$ is the local representation attached to $f$ at the prime $p$ and $I$ denote the $2 \times 2$ identity matrix.
\end{thm}

The next proposition determines the quantity $\mathscr{F} W_p f (I) $ in terms of the pseudo-eigenvalue of $f$ at $p$.
\begin{prop} \label{pseudo1}
	For a newform $f \in S_k(N, \Psi)$ and  for every $p \mid N$, we have
	\[ \mathscr{F} W_p f (I)  = e^{-2\pi}\lambda_p(f).\]
\end{prop}
\begin{proof}
	As $W_p f = \lambda_p(f) g$, we deduce that 
	\begin{equation} \label{eq1}
		\mathscr{F} W_p f (I) = \lambda_p(f) \mathscr{F} g (I) = \lambda_p(f) \int_{\Q \backslash \mathbb{A}} g\left( \begin{bmatrix} 1 & u \\ 0 & 1  \end{bmatrix} \right) \psi(-u) \, du.
	\end{equation}
	Let us consider the standard character $\theta$ of $\mathbb{A}$ which is defined by $\theta(x) = \prod_{p \leq \infty}^{} \theta_p(x_p)$, where for $x_p \in \Q_p$, we have $\theta_p(x_p) = e^{2\pi i r_p(x_p)}$ with $r_p(x_p)$ being the principal part of $x_p$, and $\theta_\infty(x)=e^{-2\pi i x}$ for $x \in \R$ (see \cite[Page $110$]{KL}).We choose $\psi$ such that $\psi(u)=\theta(-u)$. By Corollary 12.4 of \cite{KL},
\begin{equation} \label{eq2}
\int_{\Q \backslash \mathbb{A}} g\left( \begin{bmatrix} 1 & u \\ 0 & 1  \end{bmatrix} \right) \psi(-u)  du = e^{-2\pi} a_1(g).
\end{equation}
Since $g$ is a newform, $a_1(g) = 1$. Now, the result follows directly from Equations \ref{eq1} and \ref{eq2}.
\end{proof}

\begin{remark} \label{al-vs-epsilon}
    Using Theorem \ref{pseudo} and Proposition \ref{pseudo1}, we conclude that for a newform $f \in S_k(N, \Psi)$, we have $\lambda_p(f)= e^{2\pi}{p^{-N_p/2}} \Psi_p(-\frac{1}{p^{N_p}}) \varepsilon_p(\pi_p, \psi_p)$. 
\end{remark}

\subsection{Galois equivariance of Atkin-Li pseudo-eigenvalue} 
Our goal in this subsection is to determine how the Atkin-Li pseudo-eigenvalue $\lambda_p(f)$ changes under the action of $\Gal(\overline{\Q}/\Q)$. Establishing this precise equivariance allows us to define an equivalence relation (cf. relation \ref{equiv-reln-on-al-ev}) on the unit circle.  The set of Atkin-Li pseudo-eigenvalues up to this equivalence relation gives a well-defined invariant of the Galois orbit of modular form.

Let $F$ be a non-Archimedean local field with ring of integers $\mathcal{O}_F$, maximal ideal $\mathfrak{p}_F$, and normalized valuation $v_F$. Let $W_F$ denote the Weil group of $F$, and $I_F \subset W_F$ the inertia subgroup. Fix an automorphism $\sigma \in \text{Aut}(\mathbb{C})$. We freely use the notations of \cite{TateCorvallis} in this section.

\begin{prop}[Galois invariance of Artin conductor] \label{Gal-inv-cond}
Let $(\rho, V)$ be a finite-dimensional, continuous complex representation of $W_F$. Let the twisted representation $V^\sigma$ be defined by the homomorphism $\rho^\sigma(g) = \sigma(\rho(g))$. If $a(V)$ denotes the Artin conductor of $V$, then $a(V^\sigma) = a(V)$.
\end{prop}

\begin{proof}
Since $\rho(I_F)$ is finite, $a(V)$ can be computed using the ramification groups $G_i$ of a finite Galois extension:
$$ a(V) = \sum_{i \ge 0} \frac{|G_i|}{|G_0|} \dim(V/V^{G_i}). $$
Thus, it suffices to show $\dim(V^{G_i}) = \dim((V^\sigma)^{G_i})$ for all $i$. Let $d = \dim V$. The fixed space $V^{G_i}$ is the kernel of the linear map $T: V \to V^{\oplus |G_i|}$ defined by $v \mapsto \bigoplus_{g\in G_i}(\rho(g)-I_d)v$. Fixing a basis, let $B$ be the matrix representing $T$. By the Rank-Nullity theorem, $\dim(V^{G_i}) = d - \operatorname{rank}(B)$.

Similarly, $(V^\sigma)^{G_i}$ is the kernel of the twisted matrix $\sigma(B)$. Because the field automorphism $\sigma$ commutes with the determinant, a minor of $B$ vanishes if and only if the corresponding minor of $\sigma(B)$ vanishes. Hence, $\operatorname{rank}(B) = \operatorname{rank}(\sigma(B))$, which implies $\dim(V^{G_i}) = \dim((V^\sigma)^{G_i})$, and the result follows.
\end{proof}

Let $\pi_{f,p}$ be the local automorphic representation at $p$ attached to the newform $f$, and let $V$ be the two-dimensional Weil-Deligne representation of $W_{\Q_p}$ corresponding to $\pi_{f,p}$ under the local Langlands correspondence. Let $\varepsilon_L(V,\psi_p)$ and $\varepsilon_D(V,\psi_p,dx_{\psi_p})$ denote the Langlands and Deligne local constants as defined in \cite[Sec.~3.6]{TateCorvallis}. Then
\[
\varepsilon_L(V,\psi_p)=\varepsilon_D(V\omega_{1/2},\psi_p,dx_{\psi_p})
=\left(\tfrac{dx_{\psi_p}}{dx_1}\right)^2 \varepsilon_D(V\omega_{1/2},\psi_p,dx_1).
\]
Since $F=\Q_p$ is non-archimedean, $\frac{dx_1}{dx_{\psi_p}}=q_F^{-n(\psi_p)/2}$, hence $\varepsilon_L(V,\psi_p)=q_F^{n(\psi_p)}\varepsilon_D(V\omega_{1/2},\psi_p,dx_1)$. Using \cite[Eq.~3.4.5]{TateCorvallis} with $s=\tfrac12$, we get $\varepsilon_L(V,\psi_p)=\varepsilon_D(V,\psi_p,dx_1)\,q_F^{-a(V)/2}$, where $q_F$ is the cardinality of the residue field of $F$.

Now, for some $\sigma \in \Gal(\overline{\Q}/\Q)$, its action on the local additive character $\psi_p(x_p) = e^{-2\pi i r_p(x_p)}$ (where $r_p(x_p) \in \Q/\Z$ is the principal part of $x_p$) is given by $\psi_p^\sigma(x_p) = \sigma(\psi_p(x_p)) = \psi_p(ux_p)$, where $u = \chi_p(\sigma) \in \Z_p^\times$ is the $p$-adic cyclotomic character evaluated at $\sigma$. By the Galois equivariance of $\varepsilon_D$ (see \cite[Sec.~3.6]{TateCorvallis}), we have:
\[ \sigma(\varepsilon_D(V, \psi_p, dx_1)) = \varepsilon_D(V^\sigma, \psi_p^\sigma, dx_1). \]
Using \cite[Eq.~3.4.4]{TateCorvallis}, twisting the additive character by the unit $u$ yields:
\[ \varepsilon_D(V^\sigma, \psi_p^\sigma, dx_1) = \det V^\sigma(u) \, \varepsilon_D(V^\sigma, \psi_p, dx_1). \]

Under the local Langlands correspondence and local-global compatibility, the determinant of $V$ evaluated on $\Z_p^\times$ corresponds to the inverse of the local nebentypus $\Psi_p$. Therefore, the conjugated representation $V^\sigma$ has determinant $\det V^\sigma(u) = (\Psi_p^\sigma)^{-1}(u) = \Psi_p^\sigma(u^{-1})$. Substituting this into the previous equation yields the following:
\[ \sigma(\varepsilon_D(V, \psi_p, dx_1)) = \Psi_p^\sigma(u^{-1}) \, \varepsilon_D(V^\sigma, \psi_p, dx_1). \]


Combining this with Proposition \ref{Gal-inv-cond} and the comparison between Langlands and Deligne local constants gives:
\begin{equation}\label{Gal-equiv-Langlandsfactor-Galoisside}
\frac{\varepsilon_L(V^\sigma,\psi_p)}{\sigma(\varepsilon_L(V,\psi_p))}
= \Psi_p^\sigma(u) \frac{q_F^{-a(V)/2}}{\sigma(q_F^{-a(V)/2})}.
\end{equation}

Under the local Langlands correspondence, we have $\pi_{f^\sigma,p}\cong \sigma(\pi_{f,p})$, where the Galois action on representations is as defined in \cite[Sec.~7.4]{Henniart}. Moreover, as $V$ corresponds to $\pi_{f,p}$, by \cite[Propriété~3]{Henniart}, $V^\sigma$ corresponds to $\pi_{f^\sigma,p}$. Now as the local Langlands correspondence preserves $\varepsilon$-factors (cf.~\cite[Sec.~33]{BH}), one has
\[
\varepsilon_L(V,\psi_p)=\varepsilon_p(\pi_{f,p},\psi_p)
\quad \text{and} \quad
\varepsilon_L(V^\sigma,\psi_p)=\varepsilon_p(\pi_{f^\sigma,p},\psi_p).
\]
Since $q_F=p$ and $a(V)=N_p$, finally we get
\begin{equation}\label{galois-equiv-epsilon-C}
\varepsilon_p(\pi_{f^\sigma,p},\psi_p)
= \Psi_p^\sigma(\chi_p(\sigma)) \left(\frac{\sigma(\sqrt{p})}{\sqrt{p}}\right)^{N_p}
\varepsilon_p(\pi_{f,p},\psi_p)^\sigma.
\end{equation}

Therefore by Remark \ref{al-vs-epsilon}, we conclude the following:
\begin{equation} \label{Gal-equiv-al-C}
    \frac{\lambda_p(f^{\sigma})}{\sigma(\lambda_p(f))}
    = \Psi_p^\sigma(\chi_p(\sigma)) \frac{e^{2\pi}}{\sigma(e^{2\pi})} 
      \left( \frac{\sigma(\sqrt{p^{N_p}})}{\sqrt{p^{N_p}}} \right)
      \left(\frac{\sigma(\sqrt{p})}{\sqrt{p}}\right)^{N_p}.
\end{equation}

For $\sigma \in \Gal(\overline{\Q}/\Q)$, we have $\sigma(e^{2\pi})=e^{2\pi}$, hence
\begin{equation} \label{Gal-equiv-al-barQ}
    \lambda_p(f^{\sigma})
    = \Psi_p^\sigma(\chi_p(\sigma)) \left( \frac{\sigma(\sqrt{p^{N_p}})}{\sqrt{p^{N_p}}} \right)
      \left(\frac{\sigma(\sqrt{p})}{\sqrt{p}}\right)^{N_p}
      \sigma(\lambda_p(f)).
\end{equation}

Notice that the coefficient simplifies depending on the parity of $N_p$:

\begin{itemize}
    \item If $N_p$ is \textbf{even}, then $\sqrt{p^{N_p}}=p^{N_p/2}\in\Q$, so the first fractional factor is $1$, and
    \[
    \left(\frac{\sigma(\sqrt{p})}{\sqrt{p}}\right)^{N_p}=1.
    \]
    Hence the total fractional part is $1$.

    \item If $N_p$ is \textbf{odd}, then $\sqrt{p^{N_p}}=p^{(N_p-1)/2}\sqrt{p}$, so the first fractional factor is $\frac{\sigma(\sqrt{p})}{\sqrt{p}}$, and
    \[
    \left(\frac{\sigma(\sqrt{p})}{\sqrt{p}}\right)^{N_p}
    = \frac{\sigma(\sqrt{p})}{\sqrt{p}}.
    \]
    Hence the total fractional part is also $1$.
\end{itemize}

Therefore Equation \ref{Gal-equiv-al-barQ} simplifies to
\begin{equation} \label{al-gal-equiv-final}
\lambda_p(f^\sigma) = \Psi_p^\sigma(\chi_p(\sigma)) \sigma(\lambda_p(f)) \quad \text{for all } \sigma \in \Gal(\overline{\Q}/\Q).
\end{equation}

\begin{remark}
In the non-supercuspidal case, that is $a_p \neq 0$, the Galois equivariance of $\lambda_p(f)$ can also be deduced directly from \cite[Thm. 2.1]{al} via the Galois action on Gauss sums.
\end{remark}

In our setting, we restrict to the case where the global nebentypus $\Psi$ is quadratic. Its local component $\Psi_p$ therefore takes values in $\{\pm 1\}$. Consequently, $\Psi_p$ is invariant under the absolute Galois group, meaning $\Psi_p^\sigma = \Psi_p$ for all $\sigma \in \Gal(\overline{\Q}/\Q)$. The general equivariance relation thus simplifies to:
\[
\lambda_p(f^\sigma) = \Psi_p(\chi_p(\sigma)) \sigma(\lambda_p(f)).
\]

Since the local nebentypus itself is fixed across the orbit, we can define an equivalence relation $\sim$ on the unit circle $S^1$ as follows: for $z_1, z_2 \in S^1$, we declare
\begin{equation} \label{equiv-reln-on-al-ev}
    z_1 \sim z_2 \quad \text{if and only if} \quad z_2 = \Psi_p(\chi_p(\sigma)) \sigma(z_1)
\end{equation}
for some $\sigma \in \Gal(\overline{\Q}/\Q)$. The preceding discussion establishes the following result:

\begin{prop} \label{prop-al-invariant}
    Let $f \in S_k(N, \Psi)$ be a newform with a quadratic nebentypus $\Psi$, where $p$ is a prime dividing $N$, and let $\sim$ be the equivalence relation on $S^1$ defined in (\ref{equiv-reln-on-al-ev}). Then the equivalence class of the Atkin-Li pseudo-eigenvalue under $\sim$ is an invariant of the Galois orbit of $f$. That is, for any $\sigma \in \Gal(\overline{\Q}/\Q)$, we have 
    \[ \lambda_p(f^{\sigma}) \sim \lambda_p(f). \]
\end{prop}

\begin{remark} \label{al-general-nebentypus}
    When the global nebentypus $\Psi$ is not restricted to being quadratic, its local component $\Psi_p$ is generally not invariant under the absolute Galois group (i.e., $\Psi_p^\sigma \neq \Psi_p$). Consequently, the Atkin-Li pseudo-eigenvalue cannot be treated as an isolated invariant; it must be tracked alongside the conjugated nebentypus. To formalize this, let $X := \widehat{\mathbb{Z}_p^\times} \times S^1$, where $\widehat{\mathbb{Z}_p^\times}$ denotes the group of finite-order continuous characters of $\mathbb{Z}_p^\times$. We define an equivalence relation $\sim$ on $X$ by declaring $(\psi_1, z_1) \sim (\psi_2, z_2)$ if and only if there exists $\sigma \in \Gal(\overline{\mathbb{Q}}/\mathbb{Q})$ such that:
    \[ \psi_2 = \psi_1^\sigma \quad \text{and} \quad z_2 = \psi_1^\sigma(\chi_p(\sigma)) \sigma(z_1). \]
    Under this relation, the general Galois equivariance established in Equation \ref{Gal-equiv-al-barQ} immediately yields $(\Psi_p^\sigma, \lambda_p(f^\sigma)) \sim (\Psi_p, \lambda_p(f))$. Therefore, the equivalence class $[(\Psi_p, \lambda_p(f))]_\sim$ in the quotient space $X/\sim$ constitutes a well-defined invariant for the Galois orbit of $f$, generalizing our invariant beyond the quadratic nebentypus restriction.
\end{remark}
While Proposition \ref{prop-al-invariant} establishes that Atkin-Li pseudo-eigenvalue up to equivalence \ref{equiv-reln-on-al-ev} is an invariant of the Galois orbit of $f$, this invariant is generally not preserved under arbitrary twists. Recall in the trivial nebentypus case, \cite{dpt} observed that if $f \in S_k(\Gamma_0(p))$ and $\chi_p$ is a quadratic character of $p$-power conductor, then the twisted newform $f \otimes \chi_p \in S_k(\Gamma_0(p^2))$ has a predetermined Atkin-Lehner eigenvalue of $\chi_p(-1)$ \cite[Theorem 6]{al_1}. Thus the twist ``forgets'' about the original eigenvalue, eliminating the distinction between Galois orbits.

Similar phenomenon persists for forms with non-trivial nebentypus. Let $f \in S_k(N, \Psi)$, let $Q = p^{N_p}$ denote the $p$-primary part of the level, and let $\chi$ be a character of $p$-power conductor. If $\chi$ is highly ramified such that $N_p \leq a(\chi)$ and $a(\Psi_p \chi) = a(\chi)$, then by \cite[Theorem 4.1]{al}, the pseudo-eigenvalue of the twisted newform evaluates to:
\begin{equation} \label{eq:imp-of-min-twist}
 \lambda_p(f \otimes \chi) = \overline{\Psi_{N/Q}}(p^{a(\chi)}) \chi(-1) \frac{g(\Psi_p\chi)}{g(\overline{\chi})}. 
 \end{equation}

Because this pseudo-eigenvalue is entirely determined by the characters $\chi$ and $\Psi$, any dependence on the distinguishing Galois data of the original form $f$ is lost. Consequently, to construct a well-defined invariant for the Galois orbit, we must evaluate the pseudo-eigenvalue on a minimal quadratic twist. Following \cite{dpt}, we make the following definition:

\begin{definition} \label{defi:minimalAL-pseudo}
Let $f \in S_k(N,\Psi)$ be a newform with quadratic nebentypus $\Psi$, and let $p \mid N$. Define $T_f=\{f\otimes\chi\}$, where $\chi$ runs through all quadratic characters unramified outside $p$. Then exactly one of the following occurs:
\begin{enumerate}
    \item every form in $T_f$ has level at least $N$, or
    \item there is a unique form $g\in T_f$ whose level is minimal and smaller than $N$.
\end{enumerate}
In the first case, the minimal Atkin-Li pseudo-eigenvalue of $f$ at $p$ is defined to be the Atkin-Li pseudo-eigenvalue of $f$ itself. In the second case, it is defined to be the Atkin-Li pseudo-eigenvalue of the form $g$.
\end{definition}

Let $f_{\mathrm{min}}$ denote the form realizing this minimal level. Because any quadratic character $\chi$ is Galois-invariant, twisting commutes with Galois conjugation: $(f \otimes \chi)^\sigma = f^\sigma \otimes \chi$. Furthermore, by Proposition \ref{Gal-inv-cond}, the set of levels occurring in $T_f$ is identical to that in $T_{f^\sigma}$. Consequently, we find that $(f^\sigma)_{\mathrm{min}} = (f_{\mathrm{min}})^\sigma$.

Since $f_{\mathrm{min}}$ is a newform possessing the exact same quadratic nebentypus $\Psi$ as $f$, it satisfies the Galois equivariance relation established in Equation \ref{al-gal-equiv-final}. Evaluating the Atkin-Li pseudo-eigenvalue of the minimal twist of the conjugated form gives:
$$ \lambda_p((f^\sigma)_{\mathrm{min}}) = \lambda_p((f_{\mathrm{min}})^\sigma) = \Psi_p(\chi_p(\sigma)) \sigma(\lambda_p(f_{\mathrm{min}})). $$

Thus, the minimal Atkin-Li pseudo-eigenvalue enjoys the exact same Galois equivariance relation as the original pseudo-eigenvalue. Therefore, the equivalence class of the minimal pseudo-eigenvalue constitutes a well-defined invariant for the Galois orbit of $f$.

The minimal Atkin-Li pseudo-eigenvalue at $p$ of a newform $f$ is sometimes determined by the local type $\widetilde{\pi}_{f,p}$ of $f$ at $p$. The following is an analogue of \cite[Theorem 3.6]{dpt} in the case of quadratic nebentypus.

\begin{thm}\label{CountofLocalType}
	Let $\tau \in \mathcal{A}_p$ be such that $\widetilde{\tau} = \widetilde{\pi}_{f,p}$ for some newform $f \in S_k(N, \Psi)$ where $p$ is an odd prime that divides $N$ and $\Psi$ has order prime to $p$.
	\begin{enumerate}
		\item
		If $\widetilde{\tau}$ is a ramified principal series, ramified twist of Steinberg or an unramified supercuspidal type, then the pseudo-eigenvalue of the Atkin–Li operator $W_p$ is the same for all modular forms $f$ with local type $\widetilde{\tau}$ at $p$.
		
		\item 
		If $\widetilde{\tau}$ is a ramified supercuspidal type or an unramified twist of Steinberg, the pseudo-eigenvalues for modular forms with local type $\widetilde{\tau}$ at $p$ are unique up to a sign. Furthermore, these two values are interchanged when twisting by the quadratic unramified character, which preserves inertial types.
	\end{enumerate}
\end{thm}
	
\begin{proof} We consider each case separately:
	\begin{enumerate} 
		\item
		Let $\pi(\chi_1,\chi_2)\in \widetilde{\tau}$ be a principal series representation attached to a newform $f\in S_k(N,\Psi)$, where we write the fixed global level as $N=p^rM$ with $(p,M)=1$. Now the central character $\omega_{\pi_p} = \chi_1\chi_2$ evaluated at the uniformizer $p$ must equal the prime-to-$p$ part of the global nebentypus. This forces the identity $(\chi_1\chi_2)(p)=\Psi_M(p)$. Moreover, the local type $\widetilde{\tau}$ fixes the restrictions of $\chi_1,\chi_2$ to inertia. Since the local Gauss sum depends only on these restrictions to inertia, it follows that the local epsilon factor at $p$ is same for all forms with local type $\widetilde{\tau}$. By Remark \ref{al-vs-epsilon}, the Atkin–Li pseudo-eigenvalue is therefore uniquely determined within this class.
        \medskip
        \item 
Write $\pi_{f,p} \cong \mathrm{St} \otimes \psi $ with $\psi$ unramified. Evaluating the central character condition at $p$ gives $\psi(p)^2=\Psi_p^{-1}(p)$, which admits exactly two solutions $\psi(p)=\pm \alpha$. By \cite[Theorem 2.1]{al}, $\lambda_p(f)=-p^{k/2-1}/a_p(f)$. Since $a_p(f)=\psi(p)p^{k/2-1}$ by \cite[Prop. 2.8(2)]{lw}, we obtain $\lambda_p(f)=-\frac{1}{\psi(p)}$. Thus the two choices $\psi(p)=\pm\alpha$ produce two Atkin–Li pseudo-eigenvalues differing by a sign.

Now let $\pi_{f,p} \cong \mathrm{St} \otimes \mu$, where the character $\mu$ is ramified. By \cite[Table, p. 113]{MR0379375}, because the inducing characters are ramified, the local $\varepsilon$-factor of this special representation simplifies directly to the product of the $\varepsilon$-factors of its inducing characters:
\[ \varepsilon(\pi_{f,p}, \psi_p) = \varepsilon(\mu |\cdot|^{1/2}, \psi_p) \varepsilon(\mu |\cdot|^{-1/2}, \psi_p). \]
Applying the unramified twist formula $\varepsilon(\chi \theta, \psi_p) = \theta(p)^{a(\chi)+n(\psi_p)}\varepsilon(\chi, \psi_p)$ with $\chi = \mu$ and $\theta = |\cdot|^{\pm 1/2}$, we obtain:
\begin{align*}
    \varepsilon(\mu |\cdot|^{1/2}, \psi_p) &= (p^{-1/2})^{a(\mu)+n(\psi_p)} \varepsilon(\mu, \psi_p), \\
    \varepsilon(\mu |\cdot|^{-1/2}, \psi_p) &= (p^{1/2})^{a(\mu)+n(\psi_p)} \varepsilon(\mu, \psi_p).
\end{align*}
Multiplying these two terms gives $\varepsilon(\pi_{f,p}, \psi_p) = \varepsilon(\mu, \psi_p)^2$.

The local inertial type $\widetilde{\tau}$ fixes the restriction $\mu|_{\Z_p^\times}$. The central character condition evaluated at $p$ allows exactly two extensions: $\mu$ and $\mu \eta$, where $\eta$ is the unramified quadratic character satisfying $\eta(p) = -1$. Twisting by $\eta$ gives $\varepsilon(\mu \eta, \psi_p) = (-1)^{a(\mu)+n(\psi_p)} \varepsilon(\mu, \psi_p)$. Because the epsilon factor of the special representation squares this underlying term, we conclude:
\[ \varepsilon(\mathrm{St}  \otimes  \mu \eta, \psi_p) = \varepsilon(\mu \eta, \psi_p)^2 = \left[ (-1)^{a(\mu)+n(\psi_p)} \varepsilon(\mu, \psi_p) \right]^2 = \varepsilon(\mu, \psi_p)^2 = \varepsilon(\pi_{f,p}, \psi_p). \]
Thus, by Remark \ref{al-vs-epsilon}, the Atkin-Li pseudo-eigenvalues are identical for all modular forms $f$ possessing this local type.

		\medskip 
		\item 
		Let $\pi_{f, p} \in \widetilde{\tau}$ be a supercuspidal representation, that is the associated Weil representation via the local Langlands correspondence equals $\Ind_{W(K)}^{W(\Q_p)} \theta$, where $K/\Q_p$ is a quadratic extension. By the property of epsilon factors, we have
\begin{equation} \label{supepsilon} \varepsilon(\Ind_{W(K)}^{W(\Q_p)} \theta, \psi) = \varepsilon(\theta, \psi_K)= \theta(\pi)^{a(\theta)+n(\psi_K)} \cdot q^{-\frac{a(\theta)}{2}} \sum_{x \in \mathcal{O}_K^\times/U_K^{a(\theta)}} \theta^{-1}(x) \psi_K\left(\frac{x}{c}\right),\end{equation}
where $\psi_K = \psi \circ \text{Tr}_{K/\Q_p}$, $\pi$ is a uniformizer of $K$, and $c \in K^\times$ has valuation $a(\theta)+n(\psi_K)$. As discussed in Section \ref{sec:prelim}, the central character of $\pi_{f,p}$ is equal to $\theta|_{\Q_p^\times} \omega_{K/\Q_p}$. Equating it with the determinant of the Galois side, we deduce that 
\begin{equation} \label{equal}
    \theta|_{\Q_p^\times} \omega_{K/\Q_p}=(\chi_\ell^{k-1} \Psi)|_{\Q_p^\times},
\end{equation}
where the right hand side is considered as a character of $\Q_p^{\times}$ by local class field theory. Because the global space $S_k(N, \Psi)$ is fixed, evaluating the right-hand side at the uniformizer $p$ yields a constant determined entirely by the weight $k$ and the prime-to-$p$ part of the nebentypus. From this, we deduce that $\theta(p)$ is uniquely determined. Now, since $p$ can be taken as a uniformizer of the quadratic unramified extension of $\Q_p$, from the Equation \ref{supepsilon} we can see that restricting to the inertia subgroup at $p$ determines the local epsilon factors uniquely. We conclude using Remark \ref{al-vs-epsilon}. Next, assume that $K/\Q_p$ is ramified. 

For $p$ odd, $K=\Q_p(\pi)$ where $\pi=\sqrt{-p}$ or $\sqrt{-p\xi_{p-1}}$ depending on whether $p$ is the norm of an element of the extension field or not. Here $\xi_{p-1}$ denotes a primitive $(p-1)$-th root of unity. In this case, we consider $\pi$ as a uniformizer of $K$ so that we have $N_{K/\Q_p}(\pi)=-\pi^2$.  Let $g_\pi \in \Gal(\overline{\Q_p}/K)$ be an element that maps to $\pi \in K^\times$ under the reciprocity map. Then equating the determinants at $g_\pi$, we have
\begin{equation} \label{det-equn} \theta(\pi) \theta^\sigma(\pi) = \chi_\ell^{k-1}(g_\pi) \Psi(g_\pi). \end{equation} By the functoriality of the local Artin map (see (1.2.2) of \cite{TateCorvallis}), evaluating the character $\chi_\ell$ on $g_\pi$ corresponds to evaluating it on the norm $N_{K/\Q_p}(\pi) = -\pi^2$. Because $-\pi^2 \in p\Z_p^\times$ and $\chi_\ell$ is unramified at $p$, this evaluates exactly to $p$, reducing the cyclotomic factor to $\chi_\ell^{k-1}(g_\pi) = p^{k-1}$. Furthermore, we have $\theta^\sigma(\pi)=\theta(-\pi)$. Substituting these evaluations into the Equation \ref{det-equn}, it follows that $\theta(\pi^2)=\theta(-1) p^{k-1} \Psi_p(g_\pi)$. Again evaluating both sides of Equation \ref{equal} at $-1$, we get $\theta(-1)=\left(\frac{-1}{p}\right) \Psi_p(-1)$ and we deduce that $\theta(\pi^2)=\left(\frac{-1}{p}\right) p^{k-1} \Psi_p(-g_\pi)$, so it is uniquely determined.
		
		Now, choose an additive character $\psi$ of conductor $0$. Then by \cite[Lemma 1.8]{tunnell}, $n(\psi_K)=1$. As $\Psi$ has order prime to $p$, $a(\theta)$ must be even, hence $a(\theta)+n(\psi_K)$ odd. Note that $\theta(\pi^2)$ is uniquely determined, which means $\theta(\pi)$ is determined uniquely up to a sign. Thus from Equation \ref{supepsilon} we deduce that the epsilon factors of $\pi_{f,p} \in \widetilde{\tau}$ differ by a sign. Hence, the result follows from Remark \ref{al-vs-epsilon}. 
        \end{enumerate}
\end{proof}

\begin{remark}
    It is worth noting that the assumption of the global nebentypus $\Psi$ being of order prime to $p$ is utilized exclusively in the proof for the ramified supercuspidal representations (specifically, to deduce that the Artin conductor $a(\theta)$ must be even). For all other local types the conclusions regarding the uniqueness (or splitting) of the Atkin-Li pseudo-eigenvalues hold for any arbitrary nebentypus.
\end{remark}

Now again assume that $\Psi$ is quadratic. For a prime $p$ with $\val_p(N)=N_p\geq 1$, there is a natural map $\Phi$ from the set of newforms in \(S_k^{\mathrm{new}}(N), \Psi\) to
$\mathbf{LT}(p^{N_p}, \Psi)$ $\times S^1/\sim$, given by
\[
f \longmapsto ([\widetilde{\pi}_{f,p}],\ [\lambda_p]),
\]
where [$\widetilde{\pi}_{f,p}$] denotes the Galois orbit of the local inertial type of \(f\) at \(p\), and
[\(\lambda_p\)] is the equivalence class (under the equivalence relation defined in \ref{equiv-reln-on-al-ev}) of the minimal Atkin-Li pseudo-eigenvalue of $f$ at \(p\).

Theorem \ref{CountofLocalType} implies that this map is almost never surjective, since for any
local inertial type at most two Atkin-Li pseudo-eigenvalues can occur.

\begin{definition}
\label{def:compatibleAL}
Given a local-type Galois orbit \(\widetilde{\tau}\), a \emph{compatible minimal
Atkin-Li pseudo-eigenvalue} is a value
\(\lambda \in S^1\) such that Theorem \ref{CountofLocalType} does not preclude the
pair \(([\widetilde{\tau}],[\lambda])\) from being in the image of the above map.
\end{definition}

Let \(\mathbf{LO}(p^{n}, \Psi)\) denote the number of pairs \(([\widetilde{\tau}],[\lambda])\),
where \([\widetilde{\tau}]\) is a local-type Galois orbit of level \(p^{n}\), and nebentypus $\Psi$ and
\([\lambda]\) is an equivalence class of compatible minimal Atkin-Li pseudo-eigenvalue.

By Theorem \ref{CountofLocalType}, a local inertial type $\widetilde{\tau}$ that is either an unramified twist of Steinberg or ramified supercuspidal, permits exactly two compatible Atkin-Li pseudo-eigenvalues, differing by a sign: $\{\lambda, -\lambda\}$. We partition the set of such local types into two disjoint subsets based on how these two values behave under the equivalence relation defined in (\ref{equiv-reln-on-al-ev}):

\begin{itemize}
    \item $S_{\text{sym}}$ \textbf{(Symmetric Orbits):} The set of local types $\widetilde{\tau}$ for which the two allowed minimal pseudo-eigenvalues are equivalent ($\lambda \sim -\lambda$). Because the two values collapse into a single equivalence class, each type in $S_{\text{sym}}$ contributes $1$ to the $\mathbf{LO}$ count.
    
    \item $S_{\text{asym}}$ \textbf{(Asymmetric Orbits):} The set of local types $\widetilde{\tau}$ for which the two allowed minimal pseudo-eigenvalues are not equivalent ($\lambda \not\sim -\lambda$). Because the two values remain in distinct equivalence classes, each type in $S_{\text{asym}}$ contributes $2$ to the $\mathbf{LO}$ count.
\end{itemize}

\begin{remark} \label{rmk:steinberg-splitting}
    Whether an unramified twist of the Steinberg representation belongs to a symmetric or asymmetric orbit depends entirely on the value of the global nebentypus $\Psi$ at $p$. Because this local type occurs only when $p \parallel N$ and $\Psi_p$ is unramified, the equivalence condition $\Psi_p(\chi_p(\sigma)) \sigma(\lambda) = -\lambda$ simplifies as follows. The $p$-adic cyclotomic character $\chi_p$ takes values in $\Z_p^\times$. Since $\Psi_p$ is unramified, it is trivial on $\Z_p^\times$, meaning $\Psi_p(\chi_p(\sigma)) = 1$ for all $\sigma \in \Gal(\overline{\Q}/\Q)$. Thus, the equivalence condition reduces to $\sigma(\lambda) = -\lambda$.
    
    Recall that for this local type, with $\pi_{f,p} \cong \mathrm{St}\otimes \psi $, we have $\lambda_p(f)=-\frac{1}{\psi(p)}$, and the central character condition gives $\psi(p)^2=\Psi_p^{-1}(p)$. The resulting pseudo-eigenvalues depend entirely on whether $p$ splits or remains inert in the quadratic field cut out by the nebentypus:
    \begin{itemize}
        \item If $\Psi_p^{-1}(p) = 1$, the pseudo-eigenvalues are $\lambda \in \{\pm 1\}$. Since these values are rational, $\sigma(\lambda) = \lambda \neq -\lambda$ for every Galois automorphism $\sigma$. The two values are therefore inequivalent, so this unique local type belongs to $S_{\text{asym}}$. Consequently, for $n=1$, $|S_{\text{asym}}| = 1$, yielding exactly two distinct Atkin-Li orbits.
        \item If $\Psi_p^{-1}(p) = -1$, the pseudo-eigenvalues are $\lambda \in \{\pm i\}$. The absolute Galois group contains automorphisms (such as complex conjugation) mapping $i \mapsto -i$. Thus, the equivalence $\sigma(\lambda) = -\lambda$ is satisfied, placing the local type in $S_{\text{sym}}$. Consequently, for $n=1$, $|S_{\text{asym}}| = 0$, yielding exactly one Atkin-Li orbit.
    \end{itemize}
\end{remark}

\begin{thm} \label{count-of-LO}
    Let $p \neq 2$ be a prime number. Let $m$ be the largest odd divisor of $p+1$. The values of $\mathbf{LO}(p^n, \Psi)$ for newforms with a non-trivial quadratic nebentypus are given in Table \ref{table:valuesLO_nebentypus_ramified} when the local nebentypus $\Psi_p$ is tamely ramified, and in Table \ref{table:valuesLO_nebentypus_unramified} when $\Psi_p$ is unramified.
    
\begin{table}[h]
\centering
\scalebox{0.95}{
\begin{tabular}{||l|c|c||}
\hline
$n$ & $p \neq 3$ & $p = 3$ \\
\hline\hline
$1$ & $1$ & $1$ \\
\hline
$2$ & $\sigma_0\left(\frac{p-1}{2}\right) + \sigma_0(m) - 1$ & $1$ \\
\hline
$3$ & $2 + |S_{\text{asym}}|$ & $2 + |S_{\text{asym}}|$ \\
\hline
$n \ge 4 \text{ even}$ & $\sigma_0\left(\frac{p-1}{2}\right) + \sigma_0(m)$ & $2$ \\
\hline
$n \ge 5 \text{ odd}$ & $2 + |S_{\text{asym}}|$ & $4 + |S_{\text{asym}}|$ \\
\hline
\end{tabular}
}
\vspace{2mm}
\caption{Values for $\mathbf{LO}(p^n, \Psi)$ with a \textbf{tamely ramified} quadratic nebentypus $\Psi_p$.}
\label{table:valuesLO_nebentypus_ramified}
\end{table}

\begin{table}[h]
\centering
\scalebox{0.95}{
\begin{tabular}{||l|c|c||}
\hline
$n$ & $p \neq 3$ & $p = 3$ \\
\hline\hline
$1$ & $\begin{cases} 2 & \Psi_p^{-1}(p) = 1 \\ 1 & \Psi_p^{-1}(p) = -1 \end{cases}$ & $\begin{cases} 2 & \Psi_p^{-1}(p) = 1 \\ 1 & \Psi_p^{-1}(p) = -1 \end{cases}$ \\
\hline
$2$ & $\sigma_0(p-1) + \sigma_0(p+1) - 2$ & $3$ \\
\hline
$3$ & $2 + |S_{\text{asym}}|$ & $2 + |S_{\text{asym}}|$ \\
\hline
$n \ge 4 \text{ even}$ & $\sigma_0(p-1) + \sigma_0(p+1)$ & $5$ \\
\hline
$n \ge 5 \text{ odd}$ & $2 + |S_{\text{asym}}|$ & $4 + |S_{\text{asym}}|$ \\
\hline
\end{tabular}
}
\vspace{2mm}
\caption{Values for $\mathbf{LO}(p^n, \Psi)$ with an \textbf{unramified} quadratic nebentypus $\Psi_p$.}
\label{table:valuesLO_nebentypus_unramified}
\end{table}
\end{thm}

\begin{proof}
By Theorem \ref{CountofLocalType}(1), the Atkin-Li pseudo-eigenvalue is uniquely determined for principal series, ramified twists of Steinberg, and unramified supercuspidal types. Thus, each such local type yields exactly one Atkin-Li orbit, meaning their contribution to $\mathbf{LO}$ is exactly equal to their $\mathbf{LT}$ count. By Theorem~\ref{CountofLocalType}(2), unramified twists of Steinberg (which occur exclusively at $n=1$) and ramified supercuspidal types admit two pseudo-eigenvalues up to sign. Thus for any level $n$ containing such types, the total count is
\[
\mathbf{LO}(p^n,\Psi) = \mathbf{LT}(p^n,\Psi) + |S_{\text{asym}}|
\]

If $\Psi_p$ is tamely ramified, an unramified twist of Steinberg cannot occur. Splitting is induced exclusively by the ramified supercuspidal types, which exist only for odd $n \ge 3$. The Steinberg count at $n=2$ corresponds to a ramified twist, which does not split. Therefore, $\mathbf{LO} = \mathbf{LT}$ for $n \le 2$ and even $n \ge 4$. Note that for $n=2$, the sum simplifies to $\sigma_0\left(\frac{p-1}{2}\right) + \sigma_0(m) - 1$. At $n=3$, the $\mathbf{LT}$ count for ramified supercuspidal representations is $2$ for all odd primes, uniformly yielding $2 + |S_{\text{asym}}|$ orbits at $n=3$. For odd $n \ge 5$, substituting the $\mathbf{LT}$ sums from Table~\ref{table:nebentypus_ramified} yields $2 + |S_{\text{asym}}|$ for $p \neq 3$, and $4 + |S_{\text{asym}}|$ for $p=3$.

If $\Psi_p$ is unramified, the unramified twist of Steinberg appears exclusively at $n=1$. As established in Remark~\ref{rmk:steinberg-splitting}, the two pseudo-eigenvalues remain in distinct equivalence classes ($|S_{\text{asym}}| = 1$) if $\Psi_p^{-1}(p) = 1$, yielding exactly $1 + 1 = 2$ orbits. On the other hand, they merge into a single equivalence class ($|S_{\text{asym}}| = 0$) if $\Psi_p^{-1}(p) = -1$, yielding exactly $1$ orbit. For $n=2$, the contributing local types are principal series, ramified twists of Steinberg, and unramified supercuspidals. By Theorem~\ref{CountofLocalType}(1), none of these types split. Consequently, the $\mathbf{LO}$ count at $n=2$ is the sum of the $\mathbf{LT}$ row: $(\sigma_0(p-1)-1) + 1 + (\sigma_0(p+1)-2) = \sigma_0(p-1) + \sigma_0(p+1) - 2$. For even $n \ge 4$, no splitting occurs. For odd $n \ge 3$, the only type appearing is ramified supercuspidal. Notice that the $\mathbf{LT}$ count at $n=3$ is $2$ for all odd primes, uniformly yielding $2 + |S_{\text{asym}}|$ orbits at $n=3$. For odd $n \ge 5$, the base $\mathbf{LT}$ counts match the tamely ramified case, yielding $2 + |S_{\text{asym}}|$ for $p \neq 3$, and $4 + |S_{\text{asym}}|$ for $p = 3$.
\end{proof}

 \begin{remark} \label{rmk:local-orbits-parallel}
A different approach to defining invariants for ramified supercuspidal orbits was recently introduced in \cite[Section 3]{KM}. Although their analysis is restricted to trivial nebentypus, the underlying framework is expected to extend naturally to the quadratic nebentypus setting. In their formulation, Galois orbits of local representations are determined by the associated ramified quadratic extension, the Galois orbit of the restriction of the character to inertia, and the local root number. These invariants have direct counterparts in our framework: the ramified quadratic extension and the Galois orbit of the character restricted to inertia are captured by the local inertial type, while the local root number corresponds to the minimal Atkin-Li pseudo-eigenvalue. Consequently, the two approaches yield the same classification of local orbits and hence the same orbit count.

\end{remark}

    \section{Existence of local types with compatible Atkin–Li sign} \label{sec:existence-thm}
\begin{prop}
\label{CMsareBounded}
    Fix a positive integer $N$ and a Dirichlet character $\chi$ modulo $N$. The number of CM newforms in $S_k^{\mathrm{new}}(N, \chi)$ is a bounded function of $k$.
\end{prop}

\begin{proof}
    Let $K$ be an imaginary quadratic field with discriminant $D$, let $(D|\cdot)$ denote the Kronecker symbol associated to it and let $\sigma$ be a fixed complex embedding of $K$. Let $\mathfrak{m}$ be an integral ideal of $K$ with absolute norm $\mathrm{N}\mathfrak{m}$. Recall that a Hecke character $\psi$ of $K$ decomposes into a finite type $\psi_f$ and an infinity type $\psi_\infty$. 

    Following the notation of \cite{tsaknias}, we define a Hecke character $\psi$ of $K$ to be $(N, k, \chi)$-suitable if it satisfies the following conditions:
    \begin{itemize}
        \item $\psi$ is primitive,
        \item $|D|\mathrm{N}\mathfrak{m}=N$, where $\mathfrak{m}$ is the modulus of $\psi$,
        \item The infinity type of $\psi$ is $\psi_\infty = \sigma^{k-1}$,
        \item $\psi_f(m)=\chi(m)(D|m)$ for all integers $m$.
    \end{itemize}
    
    With this definition, the counting argument of \cite{tsaknias} carries over verbatim, replacing $(N,k)$-suitable Hecke characters with $(N,k,\chi)$-suitable Hecke characters to obtain a bound on the number of CM newforms of level $N$ and weight $k$, independent of $k$, exactly as in \cite[Remark 4.3]{tsaknias}.
\end{proof}
The proof of the following theorem is broadly following \cite[Theorem 4.1]{dpt}. We expand on the details for the sake of completion.
\begin{thm} \label{existence-thm}
Let $N$ be a positive integer such that either $N$ is a prime power, or $\val_q(N)$ is an odd integer $\geq 3$ for every prime $q \mid N$. For each prime $q \mid N$, let $\widetilde{\tau}_q$ be a local type of level $q^{\val_q(N)}$ coming from a newform $f \in S_{k}(N, \chi)$ and let $\lambda_q$, which is an algebraic integer with $|\lambda_q|=1$, be a compatible minimal Atkin-Li pseudo-eigenvalue for $\widetilde{\tau}_q$. Then there exists a positive integer $k_0$ such that for any $k' \geq k_0$, there exists a newform $f' \in S_{k'}(N, \chi)$ such that for all primes $q$ dividing $N$ 
\begin{itemize}
\item $\widetilde{\pi}_{f',q} \cong \widetilde{\tau}_q$ ,
\item the minimal Atkin-Li pseudo-eigenvalue of $f'$ at $q$ equals $\lambda_q$,
\item $f'$ does not have complex multiplication.
\end{itemize}
(We assume $k' \equiv k \pmod 2$ for all large weights $k'$ to ensure the non-nullity of the space $ S_{k'}(N, \chi)$.)
\end{thm}
\begin{proof}
Fix a prime $q$ dividing $N$. Also by Lemma \ref{p-part}, any modular form $f'$ of level $N$ satisfying $\widetilde{\pi}_{f',q} \cong \widetilde{\tau}_q \cong \widetilde{\pi}_{f,q}$ for all primes $q \mid N$, will have nebentypus $\chi$. We divide the proof into parts based on the nature of the local type. 

\emph{Ramified principal series at a fixed ramified prime:} In \cite{weinstein}, Weinstein has proved an asymptotic formula for the number of Hilbert modular forms (with any nebentypus) with a prescribed global inertial type $\widetilde{\tau}$ (see \cite[\S 1.1 and \S 1.2]{weinstein} for details). From Theorem \ref{thm:LTformula_nebentypus} and the restrictions on the possible values of level $N$, we see that in the ramified principal series case, we are only allowed to work with modular forms of fixed prime power level $N:=q^{N_q}$. Define the global type $\widetilde{\tau}=(\bigotimes_{p < \infty} \widetilde{\tau}_p) \otimes \widetilde{\tau}_{\infty} $ which is unramified at all finite places except $q$, meaning $\widetilde{\tau}_p$ is trivial if $p \neq q$. So if we fix $\widetilde{\tau}_{\infty}$ to be the local type of $\mathcal{D}_{k',k'-2}$, the discrete series of weight $k'$, \cite[Theorem 1.1]{weinstein} tells us the number of modular forms with prescribed global inertial type $\widetilde{\tau}$, denoted by $\#S(\widetilde{\tau})$, follows the asymptotics:$$\#S(\widetilde{\tau}) \sim Cd(\widetilde{\tau})$$ where $C$ is a constant and $d(\widetilde{\tau})=\dim(\widetilde{\tau}_q)(k'-1)$. Therefore, the number of modular forms $f' \in S_{k'}(q^{N_q}, \chi) $ grows linearly with the weight $k'$, confirming the existence of modular forms in all sufficiently large weights (maintaining the parity on weights) with a fixed local type. Now recall part (1) of Theorem \ref{CountofLocalType}, which states that when the local type is ramified principal series, it uniquely determines the Atkin-Li pseudo-eigenvalue. Hence $f'$ will have the required minimal pseudo-eigenvalue $\lambda_q$.

\emph{Supercuspidal or special at ramified primes:} As we are concerned with the existence of modular forms in arbitrary weight with fixed level and fixed nebentypus, we can make use of limit multiplicity results of Shin-Templier \cite{Shin-Templier} and Kim-Shin-Templier \cite{kst} (See footnote of page 110 of \cite{Shin-Templier}). Using the trace formula ideas, they prove \cite[Theorem 1.2]{kst} that if $G$ is
any connected reductive group over a totally real field, the number of automorphic forms
of weight $k'$ and level $N$ with prescribed local representations (which are supercuspidal at ramified primes) grows linearly with $k'$ (recall that we fix $\xi$ in loc. cit. to be $\mathcal{D}_{k', k'-2}$, the discrete series
of weight $k'$, which gives the linear growth for varying weight). Hence, for any sufficiently large weight $k'$ with the same parity as $k$, we have a newform $f' \in S_{k'}(N, \chi)$ with $\pi_{f',q}$ isomorphic to the prescribed local supercuspidal representation at ramified primes $q$. Remark \ref{al-vs-epsilon} states that, for a fixed level $N$ and nebentypus $\chi$, Atkin-Li pseudo-eigenvalue is uniquely determined by the local representation. As we obtain modular forms with prescribed local representations (not only local type), the minimal pseudo-eigenvalue of $f'$ must be $\lambda_q$ for all ramified $q$ where the local representation is supercuspidal. The Steinberg local type case can be handled similarly using the limit multiplicity formulas \cite [Thm. 6.4 or Cor. 6.5]{kst}

Thus, we establish the existence of a positive integer $k_0$  such that for any $k \geq k_0$, there exists a newform $f' \in S_{k'}(N, \chi)$ which has a fixed local type $\widetilde{\pi}_{f,q}$ and a compatible minimal pseudo-eigenvalue $\lambda_q$ at all $q \mid N$. In the course of the proof, we have seen that the number of modular forms satisfying the first two conditions grows linearly with the weight. Therefore, from Proposition \ref{CMsareBounded}, we conclude the existence of a non-CM modular form for all large enough weights satisfying the first two conditions.
\end{proof}

\begin{remark}
  Due to the limitations of the available limit multiplicity formulas, we restrict the level $N$ to be either a prime power, or to have an odd $p$-adic valuation $\val_p(N) \ge 3$ for all $p \mid N$. When $N$ has multiple prime factors, we must simultaneously control the Atkin-Li pseudo-eigenvalues at all dividing primes. Because pseudo-eigenvalues for Steinberg and ramified supercuspidal types can split, we must fix the exact local representations (not just the local inertial types) to eliminate sign ambiguity.
    
The trace formula methods of Kim-Shin-Templier \cite{kst} allow us to fix exact local representations at multiple primes simultaneously, but only if those representations belong to the discrete series (i.e., Steinberg or supercuspidal types). We cannot fix an exact principal series representation and guarantee its existence in an automorphic form. Since newforms with a quadratic nebentypus only exhibit ramified principal series types at $p$ when $\val_p(N) = 1$ or $\val_p(N)\geq 2 $ is even, we must exclude primes with exponent $1$ or even when $N$ is composite to ensure all prescribed local representations are discrete series. (See \cite[Remark 4.2]{dpt} for related discussions).
\end{remark}

Now we are ready to prove the main theorem of our paper.

\begin{proof} [{Proof of Theorem \ref{main_ineq}}]
  By Theorem \ref{existence-thm}, for large enough $k$, each compatible pair $([\widetilde{\tau}], [\lambda])$ is realized by a non-CM newform. Suppose two distinct pairs yield modular forms $f_1$ and $f_2$ that belong to the same global Galois orbit. This creates a contradiction:
  \begin{itemize}
      \item Theorem \ref{orbit-preserves-type} dictates that global conjugacy preserves local types, forcing $[\widetilde{\tau}_1] = [\widetilde{\tau}_2]$.
      \item Proposition \ref{prop-al-invariant} and subsequent discussion proves that global conjugacy preserves the minimal pseudo-eigenvalue equivalence class, forcing $[\lambda_1] = [\lambda_2]$.
  \end{itemize}

  Because both components are forced to match, the original pairs could not have been distinct. Thus, distinct pairs map injectively to distinct global orbits, establishing the lower bound for $\mathbf{NCM}(N,k,\Psi)$.
\end{proof}

\section{Computational Evidence}
\label{sec:computational_evidence}
To illustrate the counting formulas and the decomposition into symmetric and asymmetric ramified supercuspidal types, we present data from the LMFDB \cite{LMFDB}.

\subsection{Level 27 ($p=3$, $N_p=3$)}

By Theorem \ref{thm:LTformula_nebentypus}, $\mathbf{LT}(27,\Psi)=2$. This local type is a ramified supercuspidal representation, which admits exactly two Atkin-Li pseudo-eigenvalues differing by a sign. From the space $S_{35}(27,\Psi)$, one finds three non-CM Galois orbits realizing the pseudo-eigenvalues $\{\pm1,\pm i\}$:
\begin{itemize}
    \item Orbit 27.35.b.d: $\lambda_3 \in \{1,-1\}$,
    \item Orbits 27.35.b.b and 27.35.b.c: $\lambda_3=i$ and $\lambda_3=-i$.
\end{itemize}
Since $\Psi_3$ cuts out $\Q(\sqrt{-3})$, one checks that $1\sim -1$ and $i\sim -i$, so $|S_{\text{asym}}|=0$ and $\mathbf{LO}(27,\Psi)=2$. The computational data at weight $k=35$ yields $\mathbf{NCM}=3$. Guided by the generalized Maeda philosophy, one expects that for a fixed pair $(p^n,\Psi)$, the quantity $\mathbf{NCM}(p^n,k,\Psi)$ is eventually independent of $k$. If we assume this generalized conjecture holds and that the non-CM orbit count has already stabilized by weight $k=35$ (as suggested by the data available in the LMFDB), it follows that $\mathbf{NCM}(27,k,\Psi) = 3$ for all $k \ge 35$. Under these assumptions, we obtain the strict inequality $\mathbf{LO}(27,\Psi) < \mathbf{NCM}(27,k,\Psi)$ for all sufficiently large weights.
\subsection{Level 125 ($p=5$, $N_p=3$)}
Here $\mathbf{LT}(125,\Psi)=2$. From the spaces $S_{10}(125,\Psi)$ and $S_{12}(125,\Psi)$, one finds three non-CM orbits for each weight $k \in \{10, 12\}$:
\begin{itemize}
\item Orbits 125.k.b.a and 125.k.b.b: $\lambda_5 \in \{1,-1\}$,
\item Orbit 125.k.b.c: $\lambda_5 \in \{i,-i\}$.
\end{itemize}
 Since $\Psi_5$ cuts out $\Q(\sqrt{5})$, we have $1\sim -1$ and $i\sim -i$, hence $|S_{\text{asym}}|=0$ and $\mathbf{LO}(125,\Psi)=2$. The data gives $\mathbf{NCM}=3$, so $\mathbf{LO}(125,\Psi)<\mathbf{NCM}(125,k,\Psi)$.

\subsection{Level 343 ($p=7$, $N_p=3$)}
We have $\mathbf{LT}(343,\Psi)=2$. From the spaces $S_{3}(343,\Psi)$ and $S_{5}(343,\Psi)$, one finds three non-CM orbits for each weight $k \in \{3, 5\}$:
\begin{itemize}
\item Orbit 343.k.b.e: $\lambda_7 \in \{1,-1\}$,
\item Orbits 343.k.b.c and 343.k.b.d: $\lambda_7=i$ and $\lambda_7=-i$.
\end{itemize}
Since $\Psi_7$ cuts out $\Q(\sqrt{-7})$, one checks $1\sim -1$ and $i\sim -i$, hence $|S_{\text{asym}}|=0$ and $\mathbf{LO}(343,\Psi)=2$. Again $\mathbf{NCM}=3$, so $\mathbf{LO}(343,\Psi)<\mathbf{NCM}(343,k,\Psi)$.

\subsection{Level 243 ($p=3$, $N_p=5$)}

Here $\mathbf{LT}(243,\Psi)=4$, all of which are ramified supercuspidal types. Because each local type admits exactly two pseudo-eigenvalues differing by sign, this yields a theoretical total of $8$ possible Atkin-Li pseudo-eigenvalues. The space $S_{7}(243,\Psi)$ contains exactly $8$ non-CM global orbits, distributed among the following pseudo-eigenvalues:
\begin{itemize}
    \item Orbit 243.7.b.j: $\lambda_3 \in \{1,-1\}$,
    \item Orbits 243.7.b.c, 243.7.b.d, 243.7.b.f, and 243.7.b.h: $\lambda_3=i$,
    \item Orbits 243.7.b.e, 243.7.b.g, and 243.7.b.i: $\lambda_3=-i$.
\end{itemize}
Across all 8 global orbits, only four distinct pseudo-eigenvalues $\{\pm 1, \pm i\}$ appear. This indicates that exactly two of the four local types are realized at this weight, while the remaining two types are absent. 

For the two realized local types, we can explicitly determine their equivalence classes. Since $\Psi_3$ cuts out $\Q(\sqrt{-3})$, one checks that $1 \sim -1$ and $i \sim -i$. Therefore, both realized local types belong to $S_{\text{sym}}$. This guarantees that at least two local types are symmetric ($|S_{\text{sym}}| \ge 2$), which in turn bounds the possible asymmetric contribution of the remaining two unseen types to $|S_{\text{asym}}| \le 2$. Hence,
\[ \mathbf{LO}(243,\Psi) = \mathbf{LT}(243,\Psi) + |S_{\text{asym}}| \le 4 + 2 = 6. \]
Thus we obtain the strict inequality $\mathbf{LO}(243,\Psi) \le 6 < 8 = \mathbf{NCM}(243,7,\Psi)$, which verifies our lower bound theorem.

\section{Declaration}
\begin{itemize}
\item 
 \textbf{Conflict of interest}: On behalf of all the authors, the corresponding author states there is no Conflict of interest.
 \item 
 \textbf{Data availability}: The data used in this study are available in the L-functions and Modular Forms Database (LMFDB). Specific mathematical objects are referenced in the text using their standard LMFDB labels (e.g., 27.35.b.d).
\end{itemize}

\def\cprime{$'$}


\begin{thebibliography}{BDO96}

    \bibitem[AL70]{al_1}

    A.~O.~L. Atkin and J. Lehner, Hecke operators on $\Gamma \sb{0}(m)$, Math. Ann. {\bf 185} (1970), 134--160.

    
		\bibitem[AL78]{al}
		A.~O.~L. Atkin and W.~C.~W. Li, \emph{Twists of newforms and pseudo-eigenvalues
			of {$W$}-operators}, Invent. Math. \textbf{48} (1978), no.~3,  221--243.

        \bibitem[BH06]{BH}C.~J. Bushnell and G.~M. Henniart, {\it The local Langlands conjecture for $\rm GL(2)$}, Grundlehren der mathematischen Wissenschaften, 335, Springer, Berlin, 2006.
		
		\bibitem[BMM23]{bmm}
		D. Banerjee, T. Mandal and S. Mondal, Two properties of symmetric cube transfers of modular forms, J. Number Theory {\bf 275} (2025), 160--195.
        
\bibitem[B97]{Bump}
D.~W. Bump, {\it Automorphic forms and representations}, Cambridge Studies in Advanced Mathematics, 55, Cambridge Univ. Press, Cambridge, 1997.

   

         \bibitem[D69]{Deligne}
        P. Deligne, Formes modulaires et repr\'esentations $l$-adiques, in {\it S\'eminaire Bourbaki. Vol. 1968/69: Expos\'es 347--363}, Exp.\ No.\ 355, 139--172, Lecture Notes in Math., 175, Springer, Berlin.
        
		\bibitem[DPT21]{dpt}
		L.~V. Dieulefait, A.~Pacetti, and P.~Tsaknias, \emph{On the number of {G}alois
			orbits of newforms}, J. Eur. Math. Soc. (JEMS) \textbf{23} (2021), no.~8,
		2833--2860.
        
        

        \bibitem[F79]{flath}
		D.~Flath, \emph{Atkin-{L}ehner operators}, Math. Ann. \textbf{246} (1979/80),
		no.~2,  121--123.

        \bibitem[G75]{MR0379375}
		S.~S. Gelbart, {\it Automorphic forms on ad\`ele groups}, Annals of Mathematics Studies, No. 83, Princeton Univ. Press, Princeton, NJ, 1975 Univ. Tokyo Press, Tokyo, 1975.

\bibitem[GK80]{gerardinkutzko}
P. G\'erardin and P.~C. Kutzko, Facteurs locaux pour ${\rm GL}(2)$, Ann. Sci. \'Ecole Norm. Sup. (4) {\bf 13} (1980), no.~3, 349--384.

\bibitem[GM12]{Ghitza}
A. Ghitza and A.~W. McAndrew, Experimental evidence for Maeda's conjecture on modular forms, Tbil. Math. J. {\bf 5} (2012), no.~2, 55--69.

        
        \bibitem[H01]{Henniart}
        G.~M. Henniart, Sur la conjecture de Langlands locale pour ${\rm GL}_n$, J. Th\'eor. Nombres Bordeaux {\bf 13} (2001), no.~1, 167--187.

        \bibitem[HM97]{HidaMaeda}
        H. Hida and Y. Maeda, Non-abelian base change for totally real fields, Pacific J. Math. {\bf 1997}, Special Issue, 189--217.
        
        \bibitem[KL06]{KL}
		A.~Knightly and C.~Li, Traces of {H}ecke operators, Vol. 133 of
		\emph{Mathematical Surveys and Monographs}, American Mathematical Society,
		Providence, RI (2006), ISBN 978-0-8218-3739-9; 0-8218-3739-7.

        \bibitem[KM25]{KM}
A. Knightly and K. Martin, 
\textit{Counting newforms with prescribed ramified supercuspidal components}, 
preprint (2025), arXiv:2511.14587. 
Latest version available at \url{https://kimballmartin.github.io/papers/sctrace.pdf}.
		
         \bibitem[KST20]{kst}J. Kim, S.~W. Shin and N. Templier, Asymptotic behavior of supercuspidal representations and Sato-Tate equidistribution for families, Adv. Math. {\bf 362} (2020), 106955, 57 pp.

       \bibitem[L22]{joshlam}
Y.~H.~J.~Lam,
\textit{Motivic local systems on curves and Maeda's conjecture},
preprint (2022), arXiv:2211.06120.

\bibitem[LM]{LMFDB}
\textit{The LMFDB Collaboration}. (2026). The L-functions and Modular Forms Database. \url{http://www.lmfdb.org}.
		
		\bibitem[LW12]{lw}
		D.~Loeffler and J.~Weinstein, \emph{On the computation of local components of a
			newform}, Math. Comp. \textbf{81} (2012), no. 278,  1179--1200.
            
        \bibitem[M15]{Maeda}
            Y. Maeda, Maeda's conjecture and related topics, in {\it Algebraic number theory and related topics 2013}, 305--324, RIMS K\^oky\^uroku Bessatsu, B53, Res. Inst. Math. Sci. (RIMS), Kyoto.

            \bibitem[M21]{Martin}
            K. Martin, An on-average Maeda-type conjecture in the level aspect, Proc. Amer. Math. Soc. {\bf 149} (2021), no.~4, 1373--1386.

            \bibitem[MS16]{MurtySrinivas}
            M.~R. Murty and K. Srinivas, Some remarks related to Maeda's conjecture, Proc. Amer. Math. Soc. {\bf 144} (2016), no.~11, 4687--4692.

          


            
            

    \bibitem[S79]{Serre}
J.-P. Serre, {\it Local fields}, translated from the French by Marvin Jay Greenberg, 
Graduate Texts in Mathematics, 67, Springer, New York-Berlin, 1979.

\bibitem[Sh71]{Shimura}
G. Shimura, {\it Introduction to the arithmetic theory of automorphic functions}, Kan\^o{} Memorial Lectures Publications of the Mathematical Society of Japan, No. 1 No. 11, Iwanami Shoten Publishers, Tokyo, 1971 Princeton Univ. Press, Princeton, NJ, 1971.

    \bibitem[ST16]{Shin-Templier}
    S.~W. Shin and N. Templier, Sato-Tate theorem for families and low-lying zeros of automorphic $L$-functions, Invent. Math. {\bf 203} (2016), no.~1, 1--177. 

    \bibitem[T77]{TateCorvallis}
    J.~T. Tate, Number theoretic background, in {\it Automorphic forms, representations and $L$-functions (Proc. Sympos. Pure Math., Oregon State Univ., Corvallis, Ore., 1977), Part 2}, pp. 3--26, Proc. Sympos. Pure Math., XXXIII, Amer. Math. Soc., Providence, RI.
    
	\bibitem[T14]{tsaknias}
    P. Tsaknias, A possible generalization of Maeda's conjecture, in {\it Computations with modular forms}, 317--329, Contrib. Math. Comput. Sci., 6, Springer, Cham. 
		\bibitem[T83]{tunnell}
		J.~B. Tunnell, \emph{Local {$\epsilon $}-factors and characters of {${\rm
					GL}(2)$}}, Amer. J. Math. \textbf{105} (1983), no.~6,  1277--1307.
				

         \bibitem[W09]{weinstein}
        J.~S. Weinstein, Hilbert modular forms with prescribed ramification, Int. Math. Res. Not. IMRN {\bf 2009}, no.~8, 1388--1420. 
	\end{thebibliography}
\end{document}